\documentclass[11pt]{article}
\usepackage{comment}
\usepackage[all]{xy}
\usepackage[dvips]{graphics,graphicx}
\usepackage{amsfonts,amssymb,amsmath,xcolor,mathrsfs, dsfont,amstext,appendix}
\usepackage{amsbsy, amsopn, amscd, amsxtra, amsthm,authblk, enumerate}
\usepackage{upref}
\usepackage{geometry,bm}
\usepackage[displaymath]{lineno}
\usepackage{float}
\usepackage{bbm}

\usepackage[colorlinks,
            linkcolor=blue,
            anchorcolor=green,
            citecolor=blue
            ]{hyperref}

\title{Disconnection probability of Brownian motion on an annulus}
\author{Gefei Cai\thanks{Peking University.}\qquad Xuesong Fu\thanks{Xiangtan University.}\qquad Xin Sun$^*$\qquad Zhuoyan Xie$^*$}
\date{}

\newcommand{\R}{\mathbbm{R}}
\newcommand{\A}{\mathbbm{A}}
\newcommand{\D}{\mathbbm{D}}

\newcommand{\C}{\mathbbm{C}}

\newcommand{\cS}{\mathcal{S}}
\newcommand{\tcL}{\tilde{\mathcal{L}}}
\newcommand{\cD}{\mathcal{D}}
\newcommand{\E}{\mathbbm{E}}
\renewcommand{\P}{\mathbbm{P}}
\renewcommand{\S}{\mathbbm{S}}
\newcommand{\Z}{\mathbbm{Z}}
\newcommand{\hH}{\mathbbm{H}}
\newcommand{\QD}{\mathrm{QD}}

\newcommand{\Leb}{\mathrm{Leb}}

\newcommand{\QS}{\mathrm{QS}}

\newcommand{\BA}{\mathrm{BA}}
\newcommand{\BS}{\mathrm{BS}}
\newcommand{\BE}{\mathrm{BE}}
\newcommand{\QT}{\mathrm{QT}}
\newcommand{\QA}{\mathrm{QA}}
\newcommand{\LF}{\mathrm{LF}}
\newcommand{\CR}{\mathrm{CR}}

\newcommand{\tBA}{\mathrm{mBA}}
\newcommand{\RA}{\mathcal{BA}}
\newcommand{\sep}{\mathrm{sep}}
\DeclareMathOperator{\SLE}{SLE}
\DeclareMathOperator{\Mod}{Mod}

\DeclareMathOperator{\raSLE}{raSLE}

\def\cM{\mathcal{M}}
\newcommand{\disk}{\mathrm{disk}}

\newcommand{\lp}{\mathrm{loop}}
\newcommand{\sm}{\mathsf{m}}
\newcommand{\sn}{\mathsf{n}}

\newcommand{\st}{\mathbf{t}}

\newcommand{\fl}{\mathrm{f.l.}}
\newcommand{\fd}{\mathrm{h.d.}}
\newcommand{\fdd}{\mathrm{f.d.}}
\newcommand{\Weld}{\operatorname{Weld}}
\newcommand{\wt}{\widetilde}

\numberwithin{equation}{section}

\newtheorem{theorem}{Theorem}[section]
\newtheorem{definition}[theorem]{Definition}
\newtheorem{lemma}[theorem]{Lemma}

\newtheorem{corollary}[theorem]{Corollary}

\newtheorem{proposition}[theorem]{Proposition}
\newtheorem{remark}[theorem]{Remark}

\newcommand\wh[1]{\widehat{#1}}
\newcommand\ol[1]{\overline{#1}}

\begin{document}

\maketitle
\date{}

\begin{abstract}
We derive an exact formula for the probability that a Brownian path on an annulus does not disconnect the two boundary components of the annulus. The leading asymptotic behavior of this probability is governed by the disconnection exponent  obtained by Lawler-Schramm-Werner (2001) using the connection to Schramm-Loewner evolution (SLE). The derivation of our formula is based on this connection and the coupling with Liouville quantum gravity (LQG). As byproducts of our proof, we obtain a precise relation between Brownian motion on a disk stopped upon hitting the boundary and the SLE$_{8/3}$ loop measure on the disk; we also obtain a detailed description of the LQG surfaces cut by the outer boundary of stopped Brownian motion on a $\sqrt{8/3}$-LQG disk.  
\end{abstract}

\section{Introduction}

Two-dimensional (2D) Brownian motion is an extensively studied planar stochastic process. It enjoys conformal invariance and is deeply connected to the Schramm-Loewner evolution (SLE).
This  leads to the determination of non-intersection exponents and the proof of Mandelbrot's conjecture~\cite{LSW01a,LSW01b,LSW2002b,LSW02b}.
One of the most basic exponents for 2D Brownian motion is the disconnection exponent. Let $(B_t)_{t\ge0}$ be a standard 2D Brownian motion starting from the origin. Fix $\tau>0$, and consider two hitting times $T_0:=\inf\{t>0: |B_t| =1 \}$ and $T_\tau:=\inf\{t>0: |B_t| =e^{-2\pi \tau}\}$. Let $E_\tau$ be the event that the trajectory $B[T_\tau,T_0]$ does not disconnect the circles $\S^1=\{z:|z|= 1\}$ and $e^{-2\pi \tau}\S^1$. 
\begin{theorem}[Lawler-Schramm-Werner~\cite{LSW01b}]\label{thm:lsw}
    $\lim_{\tau\to \infty} -\frac{\log \P[E_\tau]}{2\pi\tau} =\frac{1}{4}$. 
\end{theorem}

Since Theorem~\ref{thm:lsw} only concerns the exponent, the precise definition of the disconnection event can vary. The main result of this paper is that for a natural variant we can exactly compute the disconnection probability. 
\begin{theorem}\label{thm:main}
 Let $\sigma_\tau=\sup\{0<t<T_0: |B_t| =e^{-2\pi \tau}\}$, and let $G_\tau$ be the event that the trajectory $B[\sigma_\tau,T_0]$  does not disconnect the circles $e^{-2\pi \tau}\S^1$ and $\S^1$. Then
\begin{equation}\label{eq:exact}
\P[G_\tau]=\frac{8\tau}{\sqrt{3}\eta(2i\tau)}\sum_{n=1}^\infty (-1)^{n-1}n\sin\Big(\frac{2\pi}{3}n\Big)\exp\Big(-\frac{2\pi}{3}n^2\tau\Big).
\end{equation}
Here $\eta(2i\tau):=e^{-\frac{\pi}{6}\tau}\prod_{n=1}^{\infty}\left(1-e^{-4\pi n\tau}\right)$ is the Dedekind eta function.
\end{theorem}

Note that as $\tau\to\infty$, ~\eqref{eq:exact} gives $\P[G_\tau]= 4\tau e^{-\frac{\pi}{2}\tau} + 8\tau e^{-\frac{5\pi}{2}\tau} + O(\tau e^{-\frac{9\pi}{2}\tau})$. In particular, we have $\lim_{\tau\to\infty}-\frac{\log \P[G_\tau]}{2\pi\tau}=\frac{1}{4}$, consistent with
Theorem~\ref{thm:lsw}. It would also be interesting to find a probabilistic explanation to the sub-leading terms of $\P[G_\tau]$.

Before making actual calculations, our belief that $\P[G_\tau]$ is exactly solvable comes from two-dimensional (2D) statistical physics.  
We can view the 2D simple random walk as a lattice model at criticality with a conformal invariant scaling limit (i.e. the 2D Brownian motion). 
Then the probability $\P[G_\tau]$ can be expressed as the limit of the ratio between two partition functions. Namely, for $\delta>0$, let $\delta\Z^2$ be the square lattice with mesh size $\delta$, and 
let $\A_{\tau,\delta}$ be the sub-lattice of $\delta \Z^2$ that are contained in the annulus $\A_\tau:=\{z:e^{-2\pi\tau}<|z|<1\}$. 
The boundary of $\A_{\tau,\delta}$ is defined as $\partial\A_{\tau,\delta}:=\{z\in\delta \Z^2\backslash\A_{\tau,\delta}:\mathrm{dist}(z,\A_{\tau,\delta})=\delta\}$.
Denote $\Omega^\delta$ to be the collection of random walk paths in  $\A_{\tau,\delta}$ with one endpoint at each boundary of $\A_{\tau,\delta}$. Then define  
\begin{equation*}
    Z_\delta=\sum_{\omega^\delta\in \Omega^\delta}\left(\frac{1}{4}\right)^{|\omega^\delta|}\quad \text{and}\quad    Z_{\delta,\mathrm{disconnect}}= \sum_{\omega^\delta\in \Omega^\delta}\left(\frac{1}{4}\right)^{|\omega^\delta|} \mathbf{1}_{\omega^\delta \textrm{ does not disconnects the two boundaries of }  \A_{\tau,\delta}},
\end{equation*}
where $|\mathcal \omega^\delta|$ is the number of steps in $\omega^\delta$.
Now we can express $\P[G_\tau]$ as (see Appendix~\ref{appendix_A} for details)
\begin{equation}\label{eq-scaling}
\P[G_\tau]=\lim_{\delta\to 0}   \frac{Z_{\delta,\mathrm{disconnect}}}{Z_{\delta}}. 
\end{equation}    
The ratio ${Z_{\delta,\mathrm{disconnect}}}\big/{Z_{\delta}}$ is reminiscent of the  probability that there exists an crossing open path  for a critical Bernoulli percolation on an annulus. For this crossing probability,  Cardy~\cite{Cardy06} predicted a beautiful formula, which was recently proved in~\cite{SXZ24} based on SLE and Liouville quantum gravity (LQG). This is our inspiration for the exact solvability in Theorem~\ref{thm:main}. Indeed, our proof of Theorem~\ref{thm:main} is again based on LQG and the connection between 2D Brownian motion and SLE$_6$. \\

\subsection{Proof strategy based on SLE and Liouville quantum gravity}
\label{subsection-proof-strategy}
We first review the philosophy behind our proof of  Theorem \ref{thm:main} based on quantum gravity on annulus.
The uniformly sampled random planar maps are the canonical discretizations of random surfaces in 2D pure quantum gravity. The precise scaling limit results as metric-measure spaces
were first constructed for the sphere case \cite{LeGall13,Mie13}, then for the disk case \cite{BM17}, annulus case \cite{LeGall-Annulus} and other topologies \cite{BM22}. These limiting objects are called the \textit{Brownian surfaces} in general. As shown in \cite{lqg-tbm1,lqg-tbm2,lqg-tbm3}, once we conformally embed the Brownian sphere or the Brownian disk onto the complex plane, the random geometry is given by LQG with $\gamma=\sqrt{8/3}$. Moreover, it was established in \cite{AHS17,cercle2021unit,AHS21} that the random fields inducing the random geometry are governed by the Liouville conformal field theory (CFT) on the sphere or the disk. 
According to the matter-Liouville-ghost decomposition of pure gravity, it is conjectured in \cite{Rem20} that when conditioning on the modulus, the random geometry of Brownian annulus under conformal embedding is given by the Liouville CFT on annulus with central charge 26, and the law of modulus has an explicit expression.
This is proved in \cite{ARS2022moduli} by combining the integrability of Liouville CFT on annulus \cite{Wu22} and the enumeration results for random planar maps.

The 2D Brownian motion can be interpreted as a conformal matter with central charge 0, see e.g. \cite[Remark A.3]{AG23a}. Consider now an independent Brownian excursion path on the Brownian annulus, restricted on the event that the path does not disconnect the two boundaries. 
The random geometry of the resulting matter-coupled Brownian annulus is again governed by the Liouville CFT with central charge 26, except that the law of modulus is changed. Actually, the Radon-Nikodym derivative between these two laws on moduli is the total mass of the Brownian excursion on annulus that does not disconnect two boundaries, which only depends on the modulus by conformal invariance. Consequently, 
the non-disconnection probability is obtained by dividing this Radon-Nikodym derivative by the partition function of the Brownian excursion on annulus.
This ratio formula is the continuum analog of~\eqref{eq-scaling}.

The only remaining unknown is the modulus of matter-coupled Brownian annulus, with the matter being a Brownian excursion restricted to the event that the path does not disconnect two boundaries.
In this case, as shown in Theorem~\ref{thm-equivalence-two-annuli} and Corollary~\ref{thm-weld-matter-BA},
the unexplored region is a Brownian disk with four boundary marked points, and the explored region turns out to be the concatenation of a Poissonian chain of Brownian disks.
Given Theorem~\ref{thm-equivalence-two-annuli} and Corollary~\ref{thm-weld-matter-BA}, we can in principle extract the law of the modulus of the matter-coupled Brownian annulus using the technique developed in \cite{ARS2022moduli}. However, in order to employ the connection between Brownian motion and SLE$_6$, when proving Theorem~\ref{thm:main}, it is more convenient to work with a Brownian motion on a Brownian disk starting from an interior marked point and ending when reaching the boundary. (We will prove Theorem~\ref{thm-equivalence-two-annuli} for its own interest after proving Theorem~\ref{thm:main}.)

Our actual proof of Theorem~\ref{thm:main} consists of two key steps.
The first step is to show that the outer boundary of the stopped Brownian motion arises as the interface under the conformal welding of a Brownian disk with four boundary marked points, the concatenation of a Poissonian chain of Brownian disks, and a smaller Brownian disk containing the starting point. This step is achieved in Section~\ref{section: BM-welding}, based on the $\SLE_6$ description of the outer boundary of the Brownian motion established in \cite{lawler2003conformal} and the conformal welding result from \cite{ASYZ24} for radial $\SLE_6$ on quantum disk. See Figure~\ref{fig:weld-BM1} for illustration.
The second step is to identify the complement of the new Brownian disk  with the matter-coupled Brownian annulus, which is done in Sections~\ref{section:proof of thm1.2} and~\ref{section: conformal restriction}. As an intermediate step to achieve this identification, we relate the outer boundary of the smaller Brownian disk with Werner's SLE$_{8/3}$ loop measure~\cite{werner-sle-loop}. This intermediate result can be formulated purely as a statement relating the stopped Brownian motion and the SLE$_{8/3}$  loop measure. We present it here for its independent interest.

Let $\D:=\{z\in\C:|z|<1\}$ be the unit disk.
For a Brownian motion starting from 0 until hitting the disk boundary $\S^1$, let $\ell$ be its outer boundary, and define $\sm$ to be the law of the boundary of the connected component of $\D\backslash\ell$ containing 0. 
Recall that the $\SLE_{8/3}$ loop measure is a canonical infinite measure  on simple loops characterized by the conformal restriction property~\cite{werner-sle-loop}, which is unique modulo a multiplicative constant. For each simply connected domain $D$ containing 0, denote $\SLE_{8/3,D}^{\lp}$ to be the restriction of the $\SLE_{8/3}$ loop measure to the loops contained in $D$ and surrounding~0. 

\begin{theorem}\label{thm:sm}\label{prop: key lemma}
The measures $\sm$ and $\SLE_{8/3,\D}^\lp$ are mutually absolutely continuous. Furthermore, their Radon-Nikodym derivative is given by
    \begin{equation}
    \label{eq-thm-sm}
        \frac{d\sm}{d\SLE_{8/3,\D}^{\lp}}(\eta)=C\frac{f(\tau)}{\eta(2i\tau)}.
    \end{equation}
Here $C>0$ is a constant,
$\eta(2i\tau)$ is the Dedekind eta function as in Theorem~\ref{thm:main}, and $f$ is given by
\begin{equation}\label{eq:f}
f(\tau):=\sum_{n\ge1}(-1)^{n-1}n\sin\left(\frac{2\pi}{3} n\right)\exp\left(-\frac{2\pi}{3}n^2\tau\right),
\end{equation}
\end{theorem}

The proof of Theorem~\ref{prop: key lemma} relies on LQG techniques, while the derivation of Theorem~\ref{thm:main} from Theorem~\ref{prop: key lemma} does not (it is instead based on the conformal restriction properties of Brownian motion), see Sections~\ref{section:proof of thm1.2} and~\ref{section: conformal restriction} respectively. 

As explained in Section~\ref{section-weld-matter-BA}, our proof of Theorem~\ref{thm:main} yields the following more precise relation between the SLE$_{8/3}$ loop and the Brownian motion on the disk. Let $(B_t)_{0\le t\le \tau_\D}$ be the Brownian motion starting from $0$ until the hitting time $\tau_\D$ of $\S^1$. For $t\in(0,\tau_{\D})$, let $\partial^o B[0,t]$ be the boundary of the unbounded component of $\C\backslash B[0,t]$. Define $t_1$ to be the largest cut time $t\in(0,\tau_\D)$ such that $\partial^o B[0,t_1]$ is a simple loop (see Figure~\ref{fig:counting}).
\begin{theorem}
    \label{prop-bdy-BM}
    For $\eta$ sampled from $\SLE_{8/3,\D}^{\lp}$, denote $A_{\eta}$ to be the annular region between $\eta$ and $\S^1$. Let $\nu_\eta$ be the Brownian excursion measure on $A_\eta$ between its two boundaries (see Section~\ref{section 2.1}), restricted to the paths that do not disconnect $\eta$ from $\S^1$. Then the law of $(\eta, W)$ sampled from $\nu_\eta(dW)\SLE_{8/3,\D}^{\lp}(d\eta)$ is a constant multiple of the law of $(\partial^oB[0,t_1], B[t_1,\tau_\D]) $. 
\end{theorem}

\begin{remark}
 The statements of Theorems~\ref{prop: key lemma} and~\ref{prop-bdy-BM} only involve the Brownian motion and the $\SLE_{8/3}$ loop measure while our proof is heavily based on LQG. We expect that Theorem~\ref{prop-bdy-BM} has a proof based on conformal restriction.   It is interesting to see if there is a more direct proof for Theorems~\ref{prop: key lemma}.
\end{remark}

\subsection{Outlook}
\begin{itemize}
    \item Our Theorem~\ref{thm:main} 
can be extended in two natural directions. Consider an independent coupling of a Brownian excursion on an annulus and a Brownian loop soup with intensity $c\in (0,1]$ restricted to the same annulus.  
In a forthcoming work, we plan to derive the exact non-disconnection probability of the union of a Brownian excursion and the loop-soup cluster (of intensity $c$) it intersects. Our Theorem~\ref{thm:main} can be viewed as the limiting case as $c\to 0$. We plan to 
use tools from this paper together with   the generalized radial restriction measures introduced in~\cite{qian2019} and further studied in~\cite{cai2025}. 
    
       Another natural extension is to compute the non-intersection probabilities of multiple Brownian excursions on an annulus. 
        The leading behavior of these probabilities were given by the Brownian intersection exponents in~\cite{LSW01b}. For the case of percolation on the annulus,  Cardy~\cite{Car02} predicted the exact formulae for the crossing formulae for polychromatic 2N-arm events. In~\cite{SXZ24}, the $N=1$ case and the counterpart for the monochromatic 2-arm event were rigorously derived. Inspired by these results, we expect that the aforementioned non-intersection probabilities for Brownian motion should admit nice exact formulae, and plan to derive them in a subsequent work. 
        
    \item  Various crossing formulae for percolation were derived by Cardy using ideas from conformal field theory (CFT). Recently, there are substantial advances in both mathematics~\cite{acsw-cle,CF24} and physics~\cite{DV10, NRJ24} that connects percolation observables to CFT. The full CFT description underlying planar percolation remains to be understood. We believe that there is a non-trivial CFT that governs a large class of geometric observables for planar Brownian motion, and its loup-soup extensions. (See~\cite{Cam16,Cam22} where a class of observables for the Brownian loop soup are encoded in a CFT.) To support our belief,  note that   
by~\eqref{eq:exact}, the annulus partition function $ \frac{\P[G_\tau]}{\tau}$ for non-disconnecting Brownian motion
admits two expansions that have potential CFT interpretations. Namely, expressing~\eqref{eq:exact} in term of $\tilde{q}=e^{-2\pi\tau}$ gives the so-called closed channel expansion
        $$
            \frac{\mathbb P[G_\tau]}{\tau}=4\cdot \frac{\sum_{k\in\mathbb{Z}}(-1)^{k-1}(3k-2)\tilde{q}^{\frac{1}{3}(3k-2)^2-\frac{1}{12}}}{\prod_{n=1}^\infty (1-\tilde{q}^{2n})}=4\cdot \frac{\sum_{k\in\mathbb{Z}}(-1)^{k-1}(3k-2)\tilde{q}^{2h_{2k-\frac{4}{3},0}}}{\prod_{n=1}^\infty (1-\tilde{q}^{2n})},
        $$
        while using Poisson summation, rewriting~\eqref{eq:exact} in term of $q = e^{-\pi/\tau}$ yields the so-called open channel expansion
        $$
             \frac{\mathbb P[G_\tau]}{\tau}=\frac{6}{\tau}\cdot \frac{\sum_{k\in\mathbb{Z}}(k+\frac{1}{6})q^{\frac{3}{2}(k+\frac{1}{6})^2-\frac{1}{24}}}{\prod_{n=1}^\infty (1-q^{n})}=\frac{6}{\tau}\cdot \frac{\sum_{k\in\mathbb{Z}}(k+\frac{1}{6})q^{h_{1,3k+2}}}{\prod_{n=1}^\infty (1-q^{n})}.
        $$
        Here, $h_{r,s}:=\frac{(3r-2s)^2-1}{24}$ corresponds to the $(r,s)$-type conformal weight in the Kac table. 
        See the discussion above Theorem~1.2 in~\cite{SXZ24} for more background. In the future, we plan to explore the CFT desciption of planar Brownian motion with guidance from planar percolation. In particular, we plan to link the imaginary DOZZ formula with non-intersecting Brownian motions as done in~\cite{acsw-cle} for percolation.
        
\end{itemize}

\medskip
\noindent\textbf{Organization of the paper.} In Section \ref{section-pre}, we provide preliminaries on Brownian motion and LQG surfaces. In Section \ref{section: BM-welding}, we give the conformal welding description for the Brownian motion decorated LQG disks. Based on this, we prove Theorem~\ref{thm:sm} in Section~\ref{section:proof of thm1.2} and  finish the proof of Theorem~\ref{thm:main} in Section \ref{section: conformal restriction}. In Section~\ref{section-weld-matter-BA}, we prove Theorems~\ref{prop-bdy-BM}  and~\ref{thm-equivalence-two-annuli}.

\medskip

\noindent{\bf Acknowledgment.} We thank Xinyi Li for helpful discussions, and Michael Aizenman for historical remarks.
G.C., X.S., and Z.X.\ were partially supported by National Key R\&D Program of China (No.\ 2023YFA1010700). G.C. was partially supported by National Key R\&D Program of China (No. 2021YFA1002700). X.F. was partially supported by CNNSF (No.12171410).

\section{Preliminaries}
\label{section-pre}
We will often use probabilistic terminologies in the context of non-probability measures. For a $\sigma$-finite measure space $(\Omega,\mathcal{F},M)$, let $X$ be an $\mathcal{F}$-measurable function. 
We call the pushforward measure $M_X:=X_*M$ the \textit{law} of $X$.
We say $X$ is \textit{sampled} from $M_X$
and write $M_X[f]=\int f(x)M_X(dx)$. \textit{Weighting} the law of $X$ by some $f(X)$ refers to the measure $\wt{M}_X$ satisfying $\frac{d\wt{M}_X}{dM_X}=f$.
\textit{Conditioning} on an event $E\in\mathcal{F}$ with $0<M[E]<\infty$ corresponds to 
the probability measure $\frac{M[\cdot\cap E]}{M[E]}$ on
$(E,\mathcal{F}_E)$, where $\mathcal{F}_E=\{A\cap E:A\in \mathcal{F}\}$. When $M$ is a finite measure, let $|M|:=M(\Omega)$ be its total mass and $M^\#:=\frac{M}{|M|}$ be the normalized probability measure. For background on the notion of \textit{disintegration} on measures, we refer the reader to \cite[Definition 2.1]{acsw-loop}.

We fix some notations on annulus and modulus. For $\tau>0$, let $\A_\tau:=\{z\in\C:|z|\in(e^{-2\pi\tau},1)\}$. For each domain $A$ with annular topology, there exists a unique $\tau>0$ such that $A$ and $\A_\tau$ are conformally equivalent. We call $\tau$ the \textit{modulus} of $A$, and denote it by $\Mod(A)$.  For a bounded simply connected domain $D$ and a Jordan loop $\eta\subset D$, we also write $\Mod(\eta,D)$ as the modulus of the annular connected component of $D\backslash\eta$.

\subsection{Brownian motion and $\SLE_6$}
\label{section 2.1}

Let $A\subset\C$ be a domain with annular topology, and denote $\partial^i A$, $\partial^o A$ be its inner and outer boundaries, respectively. We first recall the definition of the \emph{Brownian excursion} on $A$ (see e.g.~\cite[Section 5.2]{Law08} for further details). For $z\in\partial^i A$ and $w\in\partial^o A$, let $\mu^\#_A(z,w)$ be the probability measure on Brownian paths in $A$ from $z$ to $w$. When $A$ has smooth boundary, denote $H_A(z,w)$ to be the boundary Poisson kernel on $A$. Then 
\begin{equation*}
    \BE(A):=\int_{\partial^iA}\int_{\partial^oA}H_{A}(z,w)\mu_A^\#(z,w)dzdw
\end{equation*}
is called the Brownian excursion measure on $A$ (between its two boundaries). Note that $|\BE(A)|=\int_{\partial^iA}\int_{\partial^oA}H_{A}(z,w)dzdw=\frac{1}{\tau}$ for $\tau=\Mod(A)$~\cite[(9)]{Law11}. Furthermore, the Brownian excursion measure is conformally invariant. Namely, for a conformal map $f:A\to A'$, the pushforward measure of $\BE(A)$ under $f$ equals $\BE(A')$. This extends the definition of $\BE(A)$ to any annular region $A\subset\C$ via the conformal map $f:A\to\A_{\Mod(A)}$.

The Brownian excursion measure above is closely related to the last hitting decomposition of the Brownian motion. Namely, for a Brownian motion $(B_t)_{0\le t\le \tau_\D}$ starting from $0$ until hitting $\S^1$, let $\sigma_\tau=\sup\{t<\tau_\D:|B_t|=e^{-2\pi\tau}\}$ be the last hitting time of $e^{-2\pi\tau}\S^1$. Then the law of $(B_t)_{\sigma_\tau\le t\le\tau_\D}$ equals $\BE(A_\tau)^\#$. In particular, the probability $\P[G_\tau]$ from Theorem~\ref{thm:main} is equal to the probability that the Brownian excursion sampled from $\BE(\A_\tau)^\#$ does not disconnect the two boundaries of $\A_\tau$. Note that the latter is invariant if we replace $\A_\tau$ by $A$ for any $A$ conformally equivalent to $\A_\tau$. For this reason, we denote $\P[G_\tau]$ to be $P(\tau)$ in the following.

We now recall the following relation between 2D Brownian motion and $\SLE_6$ in \cite{lawler2003conformal}.
\begin{theorem}[{\cite[Theorem 9.1]{lawler2003conformal}}]\label{thm:equivalence}
Suppose $D\subset\C$ is a Jordan domain containing $0$. For a Brownian motion $(B_t)_{t\ge0}$ starting from $0$, let $\tau_D$ be the hitting time of $\partial D$. Then the outer boundary of $B[0,\tau_D]$ has the same law as the outer boundary of whole plane $\SLE_6$ from 0 to $\infty$ stopped at its first hitting time of $\partial D$.
\end{theorem}

We also need the following relation between whole-plane $\SLE_6$ and radial $\SLE_6$.
\begin{proposition}[{\cite[Theorem 3]{Jiang2017}}]\label{thm-raSLE6-last-hitting}
    Suppose $D\subset\C$ is a Jordan domain containing $0$. Let $\eta$ be a radial $\SLE_6$ in $D$ from a boundary point $z\in\partial D$ to 0. Let $\sigma=\sup\{t>0:\eta(t)\in\partial D\}$. Up to a time change, the time reversal of $\eta[\sigma,\infty)$ has the same law as whole-plane $\SLE_6$ started at 0 and stopped when first hitting $\partial D$. As a consequence, the law of $\eta(\sigma)$ is the harmonic measure on $\partial D$.
\end{proposition}
Combining Theorem~\ref{thm:equivalence} and Proposition~\ref{thm-raSLE6-last-hitting}, we obtain the following lemma, which is important for the conformal welding description of the Brownian motion decorated quantum surface in Section~\ref{section: BM-welding}.

\begin{lemma}
    \label{lem-BM-raSLE6}
    Let $D,\eta$ and the last boundary hitting time $\sigma$ be as in Proposition~\ref{thm-raSLE6-last-hitting}. Then the outer boundary of $B[0,\tau_D]$ has the same law as the outer boundary of $\eta[\sigma,\infty)$.
\end{lemma}

\subsection{Liouville fields on $\hH$ and $\C$}
\label{sec:LF}
We now briefly recall the definition of the Gaussian free field (GFF). Let $D\subset\C$ be a domain which is conformally equivalent to $\D$, an annulus, or $\C$. Let $\rho(dx)$ be a compactly supported probability measure on $D$ such that $\iint_{\C\times\C} |\log|x - y| |\, \rho(dx)\rho(dy) < \infty$.
Define $H(D; \rho)$ as the Hilbert space completion of $\{f \in C^\infty(D) : \int_D f(x)\rho(dx) = 0\}$ under the inner product  
$
\langle f, g \rangle_\nabla = (2\pi)^{-1} \int_D (\nabla f \cdot \nabla g) \, dx.
$  
Let $(f_n)_{n \geq 1}$ be an orthonormal basis for $H(D; \rho)$, and $(\alpha_n)_{n \geq 1}$ an i.i.d. sequence of standard Gaussian variables. Then $h_{D,\rho}:=\sum_{n=1}^\infty \alpha_n f_n$ converges almost surely as a random generalized function, which is called the \emph{free boundary Gaussian free field} on $D$, normalized such that $\int h_{D,\rho}(x)\rho(dx)=0$. See e.g. \cite[Section 4.1.4]{DMS14} for more background.

In particular, for $D=\hH:=\{z\in\C:\Im(z)>0\}$, we take $\rho_\hH$ to be the uniform probability measure on $\hH\cap\S^1$; for $D=\C$, we take $\rho_\C$ to be the uniform probability measure on $\S^1$. Set $|z|_+ := \max\{|z|, 1\}$ for $z \in \C$, and define the Green's function 
$$
\begin{aligned}
G_{\hH}(z, w) &= -\log|z - w| - \log|z - \bar{w}| + 2\log|z|_+ + 2\log|w|_+, &&\quad z, w \in \hH, \\
G_{\C}(z, w) &= -\log|z - w| + \log|z|_+ + \log|w|_+, &&\quad z, w \in \C.
\end{aligned}
$$ 
Then $h_{\mathcal{X},\rho_\mathcal{X}}$ defined above satisfies $\mathbb{E}[h_{\mathcal{X},\rho_\mathcal{X}}(z) h_{\mathcal{X},\rho_\mathcal{X}}(w)] = G_{\mathcal{X}}(z, w)$ for $\mathcal{X}\in\{\hH,\C\}$. We denote the law of $h_{\mathcal{X},\rho_\mathcal{X}}$ by $\P_\mathcal{X}$. 

We now recall the Liouville fields on $\C$ and $\hH$, possibly with insertions.
\begin{definition}
    Sample $(h, \textbf{c})$ from $\P_{\C} \times [e^{-2Qc} dc]$, and let $\phi = h(z) - 2Q \log |z|_+ + \textbf{c}$. Define $\LF_{\C}$ to be the law of $\phi$, whose sample is called a Liouville field on $\C$.

    Similarly, sample $(h, \textbf{c})$ from $\P_{\hH} \times [e^{-Qc} dc]$, and let $\phi = h(z) - 2Q \log |z|_+ + \textbf{c}$. Then $\LF_{\hH}$ is defined to be the law of $\phi$, and we call its sample a Liouville field on $\hH$.
\end{definition} 

\begin{definition}
    Consider parameters $(\alpha, u) \in \R \times \hH$. 
    Sample $(h,\textbf{c})$ from $\P_{\hH} \times [e^{(\alpha - Q)c} dc]$, and set $\phi(z) = h(z) - 2Q\log|z|_+ + \alpha G_{\hH}(z, u) + \textbf{c}$. The law of $\phi$ is denoted by $\LF_\hH^{(\alpha,u)}$.
\end{definition}

\begin{definition}
    Fix $m\geq 1$ and let $(\alpha_i, z_i)\in\R\times\C$ for $i=1,\dots, m$, where the $z_i$'s are distinct. Sample $(h,\textbf{c})$ from $\P_{\C} \times [e^{(\sum_{i=1}^m \alpha_i - 2Q)c} dc]$.
    Let $\phi(z) = h(z) - 2Q \log|z|_+ + \sum_{i=1}^m \alpha_i G_{\C}(z, z_i) + \textbf{c}$. Denote the law of $\phi$ by $\LF_{\C}^{(\alpha_i, z_i)_i}$, and refer to a sample from $\LF_{\C}^{(\alpha_i, z_i)_i}$ as a \emph{Liouville field} on $\C$ with insertions $(\alpha_i, z_i)_{1 \leq i \leq m}$. 
\end{definition}

We mention that $\LF_\hH^{(\alpha,u)}$ and $\LF_\C^{(\alpha_i,z_i)_i}$ can also be obtained by the reweighting $e^{\alpha\phi(u)}\LF_\hH$ or $\prod_{i=1}^m e^{\alpha_i\phi(z_i)}\LF_\C$ via regularization and limiting procedures; see e.g.~\cite[Lemma 2.6]{AHS21}.

\subsection{Quantum surfaces}
\label{sec-canonical-quantum-surface}
We first review the notion of \emph{quantum surfaces}. For $\gamma \in (0, 2)$ and $Q = \frac{2}{\gamma} + \frac{\gamma}{2}$, consider pairs $(D, h)$ where $D \subseteq \C$ is a 2D domain and $h$ is a distribution on $D$ (often a variant of the GFF). For two pairs $(D,h)$ and $(\wt D,\wt h)$, we say $(D, h) \sim_\gamma (\widetilde{D}, \widetilde{h})$ if there exists a conformal map $g: D \to \widetilde{D}$ satisfying
\begin{equation}\label{eq-quantun-surface}
    \wt h = h \circ g^{-1} + Q \log |(g^{-1})'|. 
\end{equation}  
A \emph{quantum surface} $S$ is the equivalence class $(D, h)/\mathord\sim_\gamma$, and a representative $(D, h)$ of $S$ is called an \emph{embedding} of $S$. Similarly, consider tuples $(D,h,(x_i)_{i\in\mathcal{I}},(\eta_j)_{j\in\mathcal{J}})$, where $(x_i)_{i\in\mathcal{I}}\subset \ol D$ is a collection of points and $(\eta_j)_{j\in\mathcal{J}}$ is a collection of Jordan curve on $\ol D$. Then we say $(D,h,(x_i)_{i\in\mathcal{I}},(\eta_j)_{j\in\mathcal{J}}) \sim_\gamma(\wt D,\wt h,(\wt x_i)_{i\in\mathcal{I}},(\wt \eta_j)_{j\in\mathcal{J}})$ if $\wt x_i=g(x_i)$ and $\wt\eta_j=g(\eta_j)$ for $i\in\mathcal{I}$ and $j\in\mathcal{J}$ under the conformal map $g$ in \eqref{eq-quantun-surface}. An equivalent class of such tuples is called a \emph{decorated quantum surface}, and its representative is called an \emph{embedding} as well.

According to~\cite{DS11,SW16}, for a (decorated) quantum surface $(D,h,(x_i)_{i\in\mathcal{I}},(\eta_j)_{j\in\mathcal{J}})/\mathord\sim_\gamma$, we can define its \emph{quantum area measure} $\mu_h$ to be the weak limit  
$
\mu_h = \lim_{\varepsilon \to 0} \varepsilon^{\gamma^2/2} e^{\gamma h_\varepsilon(z)} d^2z,
$
where $d^2z$ denotes Lebesgue measure on $D$ and $h_\varepsilon(z)$ averages $h$ over $\partial B(z, \varepsilon) \cap D$, and does not depend on the choice of embeddings. For $D = \hH$, one can also define the \emph{quantum boundary length measure}  
$
\nu_h = \lim_{\varepsilon \to 0} \varepsilon^{\gamma^2/4} e^{\gamma h_\varepsilon(x)/2} dx,
$  
with $h_\varepsilon(x)$ averaging $h$ over $\{x+\varepsilon e^{i\theta}:\theta\in(0,\pi)\}$.
According to the (locally) absolute continuity, these quantum area and boundary measures can be extended straightforwardly to other variants of GFF, e.g. the Liouville fields (possibly with insertions) defined in Section~\ref{sec:LF}.

Now we recall the \emph{beaded quantum surface}. Consider $(D, h)$ where $D$ is a closed set whose interior components with prime-end boundaries are homeomorphic to $\ol\D$, and $h$ is a distribution defined on each of these components. We can extend the equivalence relation $\sim_\gamma$ defined above for homeomorphisms $g: D \to \widetilde{D}$ that are conformal on interior components. A \emph{beaded quantum surface} $S$ is the equivalence class $(D, h)/\mathord\sim_\gamma$, and a representative $(D, h)$ of $S$ is called an \emph{embedding} of $S$. Similarly, one can define beaded quantum surfaces decorated with curves and points.

In the remainder of this section, we will focus on some specific types of quantum surfaces such as quantum disks, quantum triangles, and quantum spheres.
Quantum disk is a canonical type of quantum surface with disk topology introduced in~\cite{DMS14}.
In the following, we define quantum disk with one bulk marked point from the equivalent Liouville field description \cite[Theorem 3.4]{ARS21}, then define other kinds of quantum disks from it.

\begin{definition}
    We define $\QD_{1,0}$ to be the law of $(\hH, \phi, i)/\mathord\sim_\gamma$ for 
   $\phi$ sampled from $\frac{\gamma}{2\pi(Q-\gamma)^2}\LF_\hH^{(\gamma, i)}$. 
   For $\ell>0$, let $\QD_{1,0}(\ell)$ be the disintegration $\QD_{1,0}=\int_0^\infty\QD_{1,0}(\ell)d\ell$ such that samples from $\QD_{1,0}(\ell)$ have quantum boundary length $\ell$.
\end{definition}
\begin{definition}
\label{def-QD-n-marked}
    We write $A$ and $L$ as the total quantum area and the quantum boundary length respectively. 
    Suppose $(\D,h,0)$ is an embedding of a sample from $L\QD_{1,0}$. We define $\QD_{1,1}$ to be the law of $(\D,h,0,x)/\mathord\sim$ where $x$ is sampled from $\nu_h^{\#}$. We similarly define $\QD_{1,1}(\ell)$ for $\ell>0$ to be the disintegration of $\QD_{1,1}$ over the quantum boundary length.

    Let $(\D,h,0)$ be an embedding of a sample from $A^{-1}\QD_{1,0}$, then we define $\QD$ to be the law of $(\D,h)/\mathord\sim_\gamma$. 
    For integers $n\ge1$,
    let $(\D,h)$ be an embedding of a sample from $\frac{1}{(n-1)!}L^n\QD$. Sample $x_1,\cdots,x_n$ from the probability measure $(n-1)!\cdot1_{S_n}\nu_h^{\#}(dx_1)\cdots\nu_h^{\#}(dx_n)$ independently where $S_n$ is the event that $x_1,\cdots,x_n$ are ordered counterclockwise on $\partial\D$, then define $\QD_{0,n}$ to be the law of $(\D,h,x_1,\cdots,x_n)/\mathord\sim_\gamma$.
    For $1\le i\le n$, let $\ell_i:=\nu_h([x_i,x_{i+1}])$ where $[x_i,x_{i+1}]$ is the counterclockwise boundary arc from $x_i$ to $x_{i+1}$ and $x_{n+1}=x_1$. We define $\QD_{0,n}(\ell_1,\cdots,\ell_n)$ to be the disintegration 
     \begin{equation*}
          \QD_{0,n}=\int_{\R_+^n}\QD_{0,n}(\ell_1,\cdots,\ell_n)d\ell_1,\cdots,d\ell_n
     \end{equation*}
     such that samples from $\QD_{0,n}(\ell_1,\cdots,\ell_n)$ has quantum boundary length $\ell_i$ on $[x_i,x_{i+1}]$.

\end{definition}

We collect some length distribution results on $\QD_{1,0}$ and $\QD_{0,n}$ which are needed later.
\begin{lemma}[{\cite[Lemma 2.7]{ARS21}}]
\label{prop-one-point-disk-integrability}
    There exists a constant $C_\gamma> 0$, such that $|\QD_{1,0}(\ell)|=
    C_{\gamma}\ell^{-\frac{4}{\gamma^2}}$.
\end{lemma}

\begin{lemma} \label{lem-mass-QD}
     There exists a constant $C_\gamma>0$, such that $
     |\QD_{0,n}(\ell_1,\dots,\ell_n)|=C_\gamma(\ell_1+\dots +\ell_n)^{-\frac{4}{\gamma^2}-1}.
     $
\end{lemma}
\begin{proof}
By \cite[Proposition A.8]{DMS14} and \cite[Proposition 7.8]{AHS20}, we have $|\QD_{0,2}(\ell_1,\ell_2)|=C_\gamma(\ell_1+\ell_2)^{-\frac{4}{\gamma^2}-1}$.
    From Definition~\ref{def-QD-n-marked}, $|\QD_{0,1}(\ell)|=\ell^{-1}\int_0^\ell |\QD_{0,2}(\ell_1,\ell-\ell_1)|d\ell_1=C_\gamma\ell^{-\frac{4}{\gamma^2}-1}$. For $n\ge2$ and $\ell=\ell_1+\cdots+\ell_n$, let $(\D,h,1)$ be an embedding of a sample from $\QD_{0,1}(\ell)$. If $x_1=1$ and $x_2,\cdots,x_n$ are on $\partial\D$ in a counterclockwise order such that $\nu_h([x_i,x_{i+1}])=\ell_i$ for $1\le i\le n$ and $x_{n+1}=x_1$, then $(\D,h,1,x_2,\cdots,x_n)/\mathord\sim_\gamma$ has the law $\QD_{0,n}(\ell_1,\cdots,\ell_n)$ by Definition~\ref{def-QD-n-marked}. Hence we have $|\QD_{0,n}(\ell_1,\cdots,\ell_n)|=|\QD_{1,0}(\ell)|=C_\gamma\ell^{-\frac{4}{\gamma^2}-1}$.
\end{proof}

The \emph{thin quantum disk} is a beaded quantum surface obtained by the concatenation of Poisson point process of quantum disks with two marked points \cite[Section 2.4]{AHS20}. 

\begin{definition}\label{def-QD}
    For $\gamma\in(\sqrt{2},2)$,
    sample $T$ from the infinite measure $(\frac{4}{\gamma^2}-1)^{-2} \rm{Leb}_{\R_+}$, and then sample a Poisson point process $\{(u, \mathcal{D}_u)\}$ with the intensity measure $\mathbf{1}_{t \in [0, T]} dt \times \QD_{0,2}$. Then define the infinite measure $\cM_{0,2}^{\disk}(\gamma^2-2)$ to be the law of the ordered (according to the order by $u$) collection of doubly-marked  quantum disks $\{\mathcal{D}_u\}$. We call a sample from $\cM_{0,2}^{\disk}(\gamma^2-2)$ a thin quantum disk with weight $\gamma^2-2$.
    
    For a sample from $\cM_{0,2}^{\disk}(\gamma^2-2)$, we define its left (resp. right) quantum boundary length to be the sum of left (resp. right) quantum boundary lengths of the quantum disks $\{\mathcal{D}_u\}$. For $\ell_1,\ell_2>0$, let $\cM_{0,2}^\disk(\gamma^2-2)(\ell_1,\ell_2)$ be the disintegration $\cM_{0,2}^\disk(\gamma^2-2)=\int_{\R_+^2}\cM_{0,2}^\disk(\gamma^2-2)(\ell_1,\ell_2)d\ell_1d\ell_2$ such that samples from $\cM_{0,2}^\disk(\gamma^2-2)(\ell_1,\ell_2)$ have left (resp. right) quantum boundary length $\ell_1$ (resp. $\ell_2$) .
\end{definition}

Next we recall the notion of quantum triangle defined in~\cite{ASY22}. The following definition is consistent with~\cite[Definition 2.17 and Definition 2.18]{ASY22} according to~\cite[Proposition 2.18]{AHS21}.

\begin{definition}
\label{def-QT}
     Let $\QT(2,2,2):=\frac{2}{\gamma(Q-\gamma)}\QD_{0,3}$. We call a sample from $\QT(2,2,2)$ a thick quantum triangle with weight 2. Next we define quantum triangles with thin vertices for $\gamma\in(\sqrt{2},2)$.

     For $W_1, W_2, W_3 \in \{2,\gamma^2-2\}$, let $I=\{i\in\{1,2,3\}:W_i=\gamma^2-2\}$. 
     Suppose $I\neq\emptyset$, sample $(S_0, (S_i)_{i \in I})$ from
    $$
    \QT(2, 2, 2) \times \prod_{i \in I}\left(\frac{4}{\gamma^2}-1\right)\cM^{\disk}_{0,2}(\gamma^2-2).
    $$  
    Embed $S_0$ as $(\widetilde{D}, \phi, \tilde{a}_1, \tilde{a}_2, \tilde{a}_3)$. For each $i \in I$, embed $S_i$ as $(\widetilde{D}_i, \phi, \tilde{a}_i, a_i)$
     such that $\widetilde{D}_i$ are disjoint and $\widetilde{D}_i \cap \widetilde{D} = \tilde{a}_i$. 
    For each $i \notin I$, we set $a_i = \tilde{a}_i$. 
    Let $D = \widetilde{D} \cup \bigcup_{i \in I} \widetilde{D}_i$, and define $\QT(W_1, W_2, W_3)$ to be the law of $(D, \phi, a_1, a_2, a_3)/\mathord\sim_{\gamma}$. We call a sample from $\QT(W_1,W_2,W_3)$ a quantum triangle with weight $W_1,W_2,W_3$.
\end{definition}
\begin{figure}[htbp]
    \centering
    \includegraphics[width=0.5\linewidth]{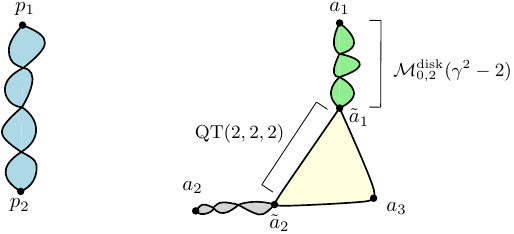}
    \caption{
    \textbf{Left:} A thin quantum disk $\cM^{\disk}_{0,2}(\gamma^2-2)$ embedded as $(D,\phi,p_1,p_2)$.
    \textbf{Right:} A quantum triangle $\QT(\gamma^2-2,\gamma^2-2,2)$ embedded as $(D,\phi,a_1,a_2,a_3)$. }
    \label{fig:weld-BM2}
\end{figure}

For $\ell_1,\ell_2,\ell_3>0$ and $W_1, W_2, W_3 \in \{2,\gamma^2-2\}$, define $\QT(W_1,W_2,W_3;\ell_1,\ell_2,\ell_3)$ by the disintegration 
    $$ \QT(W_1,W_2,W_3)=\int_{\R_+^3}\QT(W_1,W_2,W_3;\ell_1,\ell_2,\ell_3)d\ell_1d\ell_2d\ell_3,
    $$
    where a sample from $\QT(W_1,W_2,W_3;\ell_1,\ell_2,\ell_3)$ has quantum lengths $\ell_1,\ell_2,\ell_3$ for the boundary arc between the two vertices with weights $(W_1,W_2)$,  $(W_2,W_3)$, and $(W_3,W_1)$ respectively.
We also define 
$$\QT(W_1,W_2,W_3;\ell_1,\ell_2)=\int_0^\infty \QT(W_1,W_2,W_3;\ell_1,\ell_2,\ell_3)d\ell_3$$ 
and $\QT(W_1,W_2,W_3;\ell_1)=\int_{0}^\infty\QT(W_1,W_2,W_3;\ell_1,\ell_2)d\ell_2$.

The following lemma relates the quantum triangle with the thin quantum disk above.

\begin{lemma}[{\cite[Lemma 6.12]{ASY22}}]
    \label{lem-add-point}
    For $\gamma\in(\sqrt{2},2)$ and a sample from $\cM_{0,2}^{\disk}(\gamma^2-2)$, let $L$ be its left quantum boundary length.     
    Consider the quantum surface obtained by first sampled from $L\cM_{0,2}^{\disk}(\gamma^2-2)$, then sample a boundary marked point on the left boundary arc according to the probability measure proportional to the left quantum boundary length measure. Then the law of this resulting three-pointed quantum surface is a constant multiple of $\QT(2,\gamma^2-2, \gamma^2-2)$.
\end{lemma}

We end this section by recalling the quantum sphere with marked points. The following definition is also consistent with~\cite{DMS14} by~\cite[Proposition 2.26]{AHS21}.

\begin{definition}
    Fix three distinct points $z_1, z_2, z_3\in\C$, sample $\phi$ from $\frac{\pi\gamma}{2(Q-\gamma)^2}\LF_\C^{(\gamma,z_1),(\gamma,z_2),(\gamma,z_3)}$. 
    We define $\QS_3$ to be the law of $(\C,\phi,z_1,z_2,z_3)/\mathord\sim_\gamma$.
    
    Let $A$ be the total quantum area. For $(\C,\phi,z_1,z_2,z_3)$ as an embedding of a sample from $A^{-1}\QS_3$, we define $\QS_2$ to be the law of $(\C,\phi,z_1,z_2)/\mathord\sim_\gamma$.
\end{definition}

\subsection{SLE$_{8/3}$ loop and conformal welding of quantum surfaces}
\label{sec-conformal-welding}
We first recall the notion of conformal welding of quantum surfaces. See e.g. \cite[Section 3.5]{DMS14}, \cite[Section 4.1]{ASY22} and \cite[Section 2.5]{Ang24} for more details. Fix $\gamma\in(0,2)$ and $\kappa=\gamma^2$, given a certain pair of independent quantum surfaces $(\mathcal{D}_1,\mathcal{D}_2)$ and a homeomorphism identifying there boundaries, we can always topologically glue these two surfaces together to obtain a surface $(\mathcal{D},\eta)$, where $\eta$ is the gluing interface. Then $\mathcal{D}\backslash\eta$ has a conformal structure inherited from $\mathcal{D}_1$ and $\mathcal{D}_2$. The \textit{conformal welding} is a way of extending the conformal structure to the whole surface $\mathcal{D}$. In this paper, the homeomorphism is given by preserving the boundary quantum length. Since the conformal structure is local, the existence of conformal welding is due to the local absolutely continuity of quantum surfaces \cite{She16a} and the welding interface is locally absolutely continuous with respect to $\SLE_{\kappa}$. The uniqueness of conformal welding follows from the conformal removability of $\SLE_{\kappa}$ for $\kappa\in(0,4)$ and its variants~\cite{JS-removable,Rohde-Schramm-basic}. We write $\Weld(\mathcal{D}_1,\mathcal{D}_2)$ for the conformal welding of $\mathcal{D}_1$ and $\mathcal{D}_2$. Moreover, $\Weld(\mathcal{D}_1,\mathcal{D}_2)$ is measurable with respect to  $\mathcal{D}_1$ and $\mathcal{D}_2$.

We also need the notion of the \textit{uniform conformal welding} considered in \cite{AHS23-loop,acsw-loop}. Let $(\mathcal{D}_1,\mathcal{D}_2)$ be a a certain pair of independent quantum surfaces. For $i=1,2$, suppose $\mathcal{B}_i$ is a boundary component of $\mathcal{D}_i$ with the same finite total quantum length which is homeomorphism to a simple loop. As discussed above, for each $p_1\in\mathcal{B}_1$ and $p_2\in\mathcal{B}_2$, there exists a unique conformal welding of $(\mathcal{D}_1,\mathcal{D}_2)$  identifying $p_1$ and $p_2$ and preserving the boundary quantum length on $\mathcal{B}_1$ and $\mathcal{B}_2$. Now, let $\textbf{p}_1\in\mathcal{B}_1$ and $\mathbf{p}_2\in\mathcal{B}_2$ be independently sampled from the probability measure proportional to the corresponding boundary quantum length measure. Then the conformal welding of  $(\mathcal{D}_1,\mathcal{D}_2)$ identifying $\mathbf{p}_1$ and $\mathbf{p}_2$ and preserving the boundary quantum length on $\mathcal{B}_1$ and $\mathcal{B}_2$ is called the \textit{uniform conformal welding} of  $(\mathcal{D}_1,\mathcal{D}_2)$. In the following, we will also use $\Weld(\mathcal{D}_1,\mathcal{D}_2)$ to denote the uniform conformal welding of $\mathcal{D}_1$ and $\mathcal{D}_2$ in case of no ambiguity. 

The $\SLE_{8/3}$ loop measure on $\C$ is a canonical infinite measure on simple loops characterized by the conformal restriction property~\cite{werner-sle-loop}.  Namely, let $\mu$ be a measure on the simple loops on $\C$. For any simply connected $D\subset\C$, denote $\mu_D$ to be the restriction of $\mu$ to the loops contained in $D$. $\mu$ is said to satisfy \emph{conformal restriction}, if for any simply connected $D,D'\subset\C$ and any conformal map $f:D\to D'$, the pushforward of $\mu_D$ under $f$ equals $\mu_{D'}$. Then we have

\begin{proposition}[{\cite[Theorem 1]{werner-sle-loop}}]\label{prop:werner}
Suppose $\mu$ is a measure on the simple loops on $\C$ satisfying conformal restriction. Then up to a multiplicative constant, $\mu$ is equal to the $\SLE_{8/3}$ loop measure.
\end{proposition}

In particular, as pointed out in~\cite{werner-sle-loop}, the $\SLE_{8/3}$ loop measure can be obtained by taking outer boundaries from the Brownian loop measure on $\C$. (See~\cite{zhan-loop} for the construction of the $\SLE_\kappa$ loop measure with $\kappa\in(0,8)$ and~\cite{baverez2024cftsleloopmeasures} for its uniqueness with $\kappa\in(0,4]$.)

The following proposition shows that the conformal welding of two independent quantum disks gives an SLE-decorated quantum surface. Let $\SLE_{8/3}^{\sep}$ be the restriction of the $\SLE_{8/3}$ loop measure on $\C$ to the loops that separate 0 and $\infty$.

\begin{proposition}[{\cite[Proposition 6.5]{acsw-loop}}]
    \label{prop-weld-loop}
    Fix $\gamma=\sqrt{8/3}$. Let $(\C,\phi,0,\infty)$ be an embedding of a sample from $\QS_2$. Let $\eta$ be independently sampled from $\SLE^{\sep}_{8/3}$. Then there exists a constant $C>0$, such that the law of the decorated quantum surface $(\C,\phi,\eta,0,\infty)/\mathord\sim_\gamma$ is given by
    \begin{equation*}
    C\int_{0}^\infty\ell\Weld(\QD_{1,0}(\ell),\QD_{1,0}(\ell))d\ell.
    \end{equation*}
\end{proposition}

\section{Brownian motion from conformal welding}
\label{section: BM-welding}

In this section we fix $\gamma=\sqrt{8/3}$.
Recall that $(B_t)_{t\ge0}$ is a Brownian motion starting from 0 and $\tau_{\D}$ is its hitting time on $\S^1$. Let $\mathsf{P}$
be the law of the outer boundary $\ell$ of $(B_t)_{t\in[0,\tau_{\D}]}$. For $L>0$, let $(\D,\phi,0)$ be an embedding of an independent sample from $\QD_{1,0}(L)$. We write $\QD_{1,0}(L)\otimes\mathsf{P}$ for the law of the curve decorated quantum surface $(\D,\phi,\ell,0)/\mathord\sim_\gamma$. Note that it does not depend on choices of embeddings of $(\D,\phi,0)$ due to the conformal invariance of $\mathsf{P}$.

The main result in this section is the following conformal welding description of $\QD_{1,0}(L)\otimes\mathsf{P}$, which is an important ingredient in the proof of Theorem~\ref{prop: key lemma} in Section~\ref{section:proof of thm1.2}.
\begin{theorem}\label{thm:weld-BM1}
    There exists a constant $C>0$, such that
\begin{equation}
\label{eq:weldBM1}
\QD_{1,0}(L)\otimes\mathsf{P}=C\int_{\R_+^3}\Weld\left(\QD_{1,1}(b),\QD_{0,4}(a,b,c,L),\cM^{\disk}_{0,2}\left(\frac23\right)(a,c)\right)dadbdc
\end{equation}
Here $\Weld\left(\QD_{1,1}(b),\QD_{0,4}(a,b,c,L),\cM^{\disk}_{0,2}\left(\frac23\right)(a,c)\right)$ denotes the law of the curve decorated quantum surfaces obtained by conformally welding a triple of quantum surfaces sampled from $\QD_{1,1}(b)\times\QD_{0,4}(a,b,c,L)\times\cM^{\disk}_{0,2}(\frac{2}{3})(a,c)$.
See Figure \ref{fig:weld-BM1}.

\end{theorem}
\begin{figure}[htbp]
    \centering
    \includegraphics[width=0.6\linewidth]{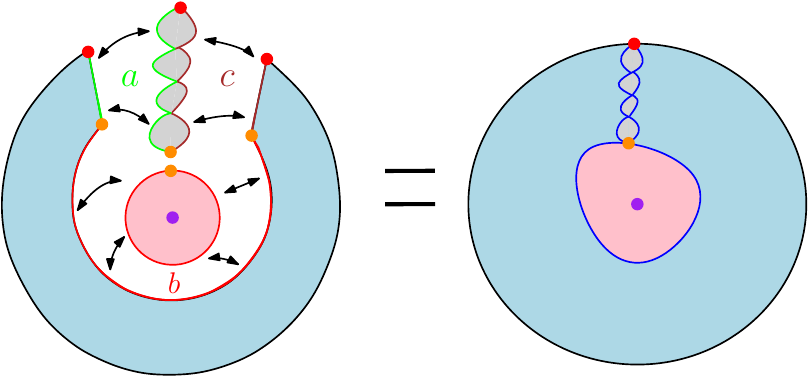}
    \caption{Illustration of Theorem \ref{thm:weld-BM1}. \textbf{Left:} samples of $\QD_{1,1}(b)$, $\QD_{0,4}(a,b,c,L)$ and $\cM^{\disk}_{0,2}\left(\frac23\right)(a,c)$. \textbf{Right:} the blue interface is the outer boundary of $(B_t)_{0\le t\le\tau_{\D}}$.}
    \label{fig:weld-BM1}
\end{figure}

In order to prove Theorem~\ref{thm:weld-BM1}, we first recall some backgrounds of forested quantum surfaces in Section~\ref{section-generalized-quantum-surface}, and review related conformal welding results in Section~\ref{section-welding-forested-surface}. Then we prove Theorem~\ref{thm:weld-BM1} in Section~\ref{section-proof-of-welding}, based on the equivalence between Brownian motion and radial $\SLE_6$ (see the end of Section~\ref{section 2.1}) as well as the conformal welding description for radial $\SLE_6$~\cite{ASYZ24}.

\subsection{Forested lines and forested quantum surfaces}
\label{section-generalized-quantum-surface}

We first recall the definition of forest lines. Let $(X_t)_{t\ge0}$ be a stable L\'evy process of index $\frac{4}{\gamma^2}$ with only upward jumps starting from 0. 
Let $G:=\{(s,X_s)_{s\ge0}\}\cup\{(s,x):s\ge0,x\in[X_{s-},X_s]\}$.
For $s<t$ and $(s,x),(t,x)\in G$, let $(s,x)\sim(t,x)$ if the horizontal line segments connecting $(s,x)$ and $(t,x)$ stays below the graph $X|_{[s,t]}$. We also set $(t,X_{t-})\sim(t,X_t)$ for each time $t$ at which $X$ jumps. We call the quotient of $G$ under the equivalence relation $\sim$ to be the \textit{loop-tree} corresponding to $X$, see e.g.~\cite{CK13looptree} for more details.
The \textit{forested line} is then a beaded quantum surface defined via the following procedure. 
Let $o=(0,0)$ be the \textit{root} of the loop-tree. 
For each jump in $X$ with size $\ell$, we sample an independent quantum disk from $\QD(\ell)^\#$ and topologically glue
the quantum disk onto the corresponding loop in a clockwise length-preserving way, where the rotation is uniformly chosen. Then the output beaded quantum surface is called a forested line, and denoted by $\mathcal{L}^o$. See \cite{DMS14,MSW-non-simple-2021,AHSY23} for more details.

The closure of the collection of the points on the boundaries of the quantum disks is called the \textit{forested boundary arc}, while the set of the points corresponding to the running infimum on the graph of $(X_t)_{t\ge0}$ is called the \textit{line boundary }\textit{arc}. For $s>0$, let $p_s$ be the point on the line boundary arc at which $X$  first takes the value $-s$, then the \textit{quantum length} between $o$ and $p_s$ is defined to be $s$. For any two points on the forested boundary arc, the \textit{generalized quantum length} between these points is defined to be the length of the corresponding time interval of $(X_t)_{t\ge0}$.

For $t>0$, the truncation of $\mathcal{L}^o$ at quantum length $t$ is the union of the line boundary arc and the quantum disks on the forested boundary arc between $o$ and $p_t$. Denote the truncation of $\mathcal{L}^o$ at quantum length $t$ by $\mathcal{L}_t$. We call the beaded quantum surface $\mathcal{L}_t$ a \textit{forested line segment}.
\begin{definition}
    \label{def-f.l.}
Let $\st$ be sampled from $\Leb_{\R_+}$, and truncate an independent forested line at quantum length $\st$.
Define $\cM^{\fl}_2$ to be the law of the resulting beaded quantum surface.
\end{definition}

We can also disintegrate the measure $\cM^{\fl}_2$ over the quantum length and generalized quantum length:
$\cM^{\fl}_2=\int_{\R_+^2}\cM^{\fl}_2(t,\ell)dtd\ell$,
where $\cM^{\fl}_2(t,\ell)$ is a measure on forested line segments with quantum length $t$ and generalized quantum length $\ell$. Let $\cM^{\fl}_2(\ell):=\int_0^\infty\cM^{\fl}_2(t,\ell)dt$. 
Next, we recall some forested quantum surfaces appeared in \cite{AHSY23,ASYZ24}.

\begin{definition}
    \label{def-f.d.}
For $\ell,t>0$, let $(\mathcal{L},\mathcal{D})$ be sampled from
$\int_{0}^\infty \cM^{\fl}_2(s,\ell)\times\cM^{\disk}_{0,2}(\frac{2}{3})(s,t)ds$, and glue the line boundary arc of $\mathcal{L}$ to the left boundary arc of $\mathcal{D}$ according to quantum length. Let $\cM^{\fd}_{0,2}(\frac{2}{3})(\ell,t)$ be the law of resulting beaded quantum surface. We also define $\cM^{\fd}_{0,2}(\frac{2}{3})(\ell):=\int_0^\infty\cM^{\fd}_{0,2}(\frac{2}{3})(\ell,t)dt$ and $\cM^{\fd}_{0,2}(\frac{2}{3}):=\int_{0}^\infty\cM^{\fd}_{0,2}(\frac{2}{3})(\ell)d\ell$. We call a sample from $\cM^{\fd}_{0,2}(\frac{2}{3})$ a half forested quantum disk.

\end{definition}

\begin{definition}
\label{def-forested-triangle}
    For $W_1,W_2,W_3\in\{\frac{2}{3},2\}$ and $\ell_1,\ell_2,t_3>0$, let $(\mathcal{L}^1,\mathcal{D},\mathcal{L}^2)$ be sampled from
$$
\int_{\R_+^2}\cM^{\fl}_2(t_1,\ell_1)\times\QT(W_1,W_2,W_3;t_1,t_2,t_3)\times\cM^{\fl}_2(t_2,\ell_2)dt_1dt_2,
$$
    and then glue the line boundary arc of $\mathcal{L}^1$ and $\mathcal{L}^2$ to the boundary arcs of $\mathcal{D}$ with quantum length $t_1$ and $t_2$ according to quantum length respectively. Define $\widetilde{\QT}(W_1,W_2,W_3;\ell_1,\ell_2,t_3)$ to be the law of resulting beaded quantum surface. We also write $\widetilde{\QT}(W_1,W_2,W_3;\ell_1,\ell_2):=\int_0^\infty\widetilde{\QT}(W_1,W_2,W_3;\ell_1,\ell_2,t_3)dt_3$ and $\widetilde{\QT}(W_1,W_2,W_3):=\int_{\R_+^2}\widetilde{\QT}(W_1,W_2,W_3;\ell_1,\ell_2)d\ell_1d\ell_2$. We call a sample from $\widetilde{\QT}(W_1,W_2,W_3)$ a forested quantum triangle.
\end{definition}

The following lemma relates the half forested quantum disk with the forested quantum triangle.
\begin{lemma}
    \label{lem-f.d.=f.q}
    For a sample from $\cM^{\fd}_{0,2}(\frac{2}{3})$, let $\ell$ be its generalized quantum length. For a beaded quantum surface sampled from $\ell\cM^{\fd}_{0,2}(\frac{2}{3})$, we sample a marked point on the forested boundary arc according to the probability measure proportional to the generalized quantum length measure.
    Then the law of the resulting quantum surface is a constant multiple of $\wt\QT(\frac{2}{3},\frac{2}{3},\frac{2}{3})$.
\end{lemma}
\begin{proof}
    This follows from \cite[Lemma 4.1]{ASYZ24}, except that in Definition \ref{def-f.d.} and \ref{def-forested-triangle}, one boundary arc is not glued to the forested line segments. 
\end{proof}

\subsection{Radial SLE$_6$ and conformal welding of forested quantum surfaces}
\label{section-welding-forested-surface}

Let $\kappa'=16/\gamma^2=6$.
According to \cite{DMS14}, if we draw an independent $\SLE_{\kappa'}$-type curve $\eta$ on a certain $\gamma$-LQG surface $\mathcal{D}$, then $\mathcal{D}$ is cut into two independent forested quantum surfaces $\mathcal{D}_1$ and $\mathcal{D}_2$. Moreover, under this coupling, the pair $(\mathcal{D},\eta)$ is measurable with respect to $(\mathcal{D}_1,\mathcal{D}_2)$. This measurable function is also called the \textit{conformal welding} of $\mathcal{D}_1$ and $\mathcal{D}_2$ (see e.g.~\cite[Section 1.2]{AHSY23} for more details). We first recall the conformal welding of two independent forested line segments.
Note that $\cM_{0,2}^{\disk}(2-\frac{\gamma^2}{2})$ is defined in Definition~\ref{def-QD} since $2-\frac{\gamma^2}{2}=\gamma^2-2=\frac{2}{3}$ when $\gamma=\sqrt{8/3}$.
\begin{proposition}[{\cite[Proposition 3.25]{AHSY23}}]
    \label{thm-weld-f.l.s} 
    Let $\cD$ be sampled from $\cM_{0,2}^{\disk}(2-\frac{\gamma^2}{2})$, and let $\tilde\eta$ be the concatenation of independent $\SLE_{\kappa'}(\frac{\kappa'}{2}-4;\frac{\kappa'}{2}-4)$ curves on each bead of $\cD$. Then there exists a constant $C>0$, such that $\tilde\eta$ divides $\cD$ into two forested lines segments $\tcL_-,\tcL_+$, whose law is
\begin{equation}
    \label{eq-weld-f.l.s}
    C\int_0^\infty\cM^{\fl}_2(\ell)\times\cM^{\fl}_2(\ell)d\ell.
\end{equation}
    Moreover, $\tcL_{\pm}$ a.s. uniquely determine $(\cD,\tilde\eta)$ in the sense that $(\cD,\tilde\eta)$ is measurable w.r.t. the $\sigma$-algebra generated by $\tcL_{\pm}$.
\end{proposition}

Next, we recall the conformal welding result for radial $\SLE_{6}$ shown in \cite{ASYZ24}. 
For $\ell>0$ and a sample from $\wt{\QT}(\frac{2}{3},\frac{2}{3},\frac{2}{3};\ell,\ell)$, we can conformally weld the two forested boundary arcs with generalized quantum length $\ell$ together. This yields a single curve-decorated quantum surface, and we denote its law by $\Weld(\wt{\QT}(\frac{2}{3},\frac{2}{3},\frac{2}{3};\ell,\ell))$. Let $(\D,\phi,0,1)$ be an embedding of a sample from $\QD_{1,1}$, and independently sample a radial $\SLE_6$ curve $\eta$ on $\D$ from 1 to 0. We denote the law of $(\D,\phi,\eta,0,1)/\mathord\sim_\gamma$ by $\QD_{1,1}\otimes\raSLE_6$.

\begin{theorem}[{\cite[Theorem 3.1]{ASYZ24}}]
    \label{thm-weld-raSLE6} 
There exists a constant $C>0$, such that
\begin{equation}
    \label{eq-weld-raSLE6}
    \QD_{1,1}\otimes\raSLE_6=C\int_0^\infty\Weld\left(\wt{\QT}\left(\frac{2}{3},\frac{2}{3},\frac{2}{3};\ell,\ell\right)\right)d\ell.
\end{equation}
\end{theorem}

We will also need the following result.

\begin{lemma}
    \label{lem-weld-two-f.d.}
Fix $\ell,t>0$. Let $(\widetilde{\mathcal{D}}_1,\widetilde{\mathcal{D}}_2)$ be sampled from $\int_0^\infty\wt{\QT}(\frac{2}{3},\frac{2}{3},\frac{2}{3};\ell,\ell')\times\cM^{\fd}_{0,2}(\frac{2}{3})(\ell',t)d\ell'$, and
conformally weld the forested boundary arcs of $\widetilde{\mathcal{D}}_1$ and $\widetilde{\mathcal{D}}_2$ with length $\ell'$ according to their generalized quantum lengths, and forget the welding interface. Then there exists a constant $C>0$, such that the law of the resulting forested quantum surface is equal to the following: first sample $(\cD,\mathcal{L})$ from
$
C\int_0^\infty\QT\left(\frac{2}{3},2,2;t',t\right)\times \cM_2^{\fl}(t',\ell)dt'
$, and then glue the line boundary arc of $\mathcal{L}$ to the boundary arc of $\cD$ with quantum length $t'$.
\end{lemma}
\begin{proof}
    This is by combining Proposition \ref{thm-weld-f.l.s}, \cite[Theorem 2.2]{AHS20} and \cite[Theorem 1.1]{ASY22}.
\end{proof}

\begin{figure}[htbp]
    \centering
    \includegraphics[width=0.7\linewidth]{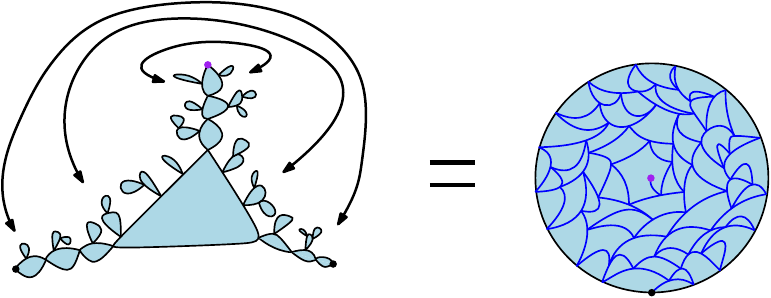}
    \caption{Illustration of Theorem~\ref{thm-weld-raSLE6}.  
    }
    \label{fig:weld-BM2}
\end{figure}

\subsection{Proof of Theorem \ref{thm:weld-BM1}}
\label{section-proof-of-welding}

Recall the setup in Theorem~\ref{thm-weld-raSLE6}.
Let $(\D,\phi,0,1)$ be an embedding of a sample from $\QD_{1,1}$, and let $\eta$ be an independent radial $\SLE_6$ in $\D$ from 1 to 0. Let $\sigma$ be the last time that $\eta$ hits $\S^1$. In the following, we denote $\raSLE_6^\sigma$ to be the law of $\eta[\sigma,\infty)$, and denote $\QD_{1,1}\otimes\raSLE_6^\sigma$ to be the law of curve-decorated quantum surface $(\D,\phi,\eta[\sigma,\infty),0,1)/\mathord\sim_\gamma$.

For $\ell>0$ and a sample from $\widetilde{\QT}(2,\frac{2}{3},2;\ell,\ell)$, let $L$ be the quantum length of its unforested boundary arc. For a sample from $L\widetilde{\QT}(2,\frac{2}{3},2;\ell,\ell)$, we sample a marked point on the unforested boundary arc according to the probability measure proportional to the quantum length measure. Then we denote the law of the resulting beaded quantum surface by $\widetilde{\QT}_{\bullet}(2,\frac{2}{3},2;\ell,\ell)$.
For a sample from $\widetilde{\QT}_\bullet(2,\frac{2}{3},2;\ell,\ell)$, we can conformally weld the two forested boundary arcs according to their generalized quantum lengths. Let $\Weld(\widetilde{\QT}_\bullet(2,\frac{2}{3},2;\ell,\ell))$ be the law of the resulting curve-decorated quantum surface.

We first give the following description of the quantum surface $\QD_{1,1}\otimes\raSLE_6^\sigma$.
\begin{proposition}
    \label{prop-weld-last-hit-raSLE}
    There exists a constant $C>0$ such that
    \begin{equation*}
    \QD_{1,1}\otimes\raSLE_6^\sigma=C\int_0^\infty\Weld\left(\widetilde{\QT}_\bullet\left(2,\frac{2}{3},2;\ell,\ell\right)\right)d\ell.
    \end{equation*}
\end{proposition}

\begin{figure}[htbp]
    \centering
    \includegraphics[width=0.9\linewidth]{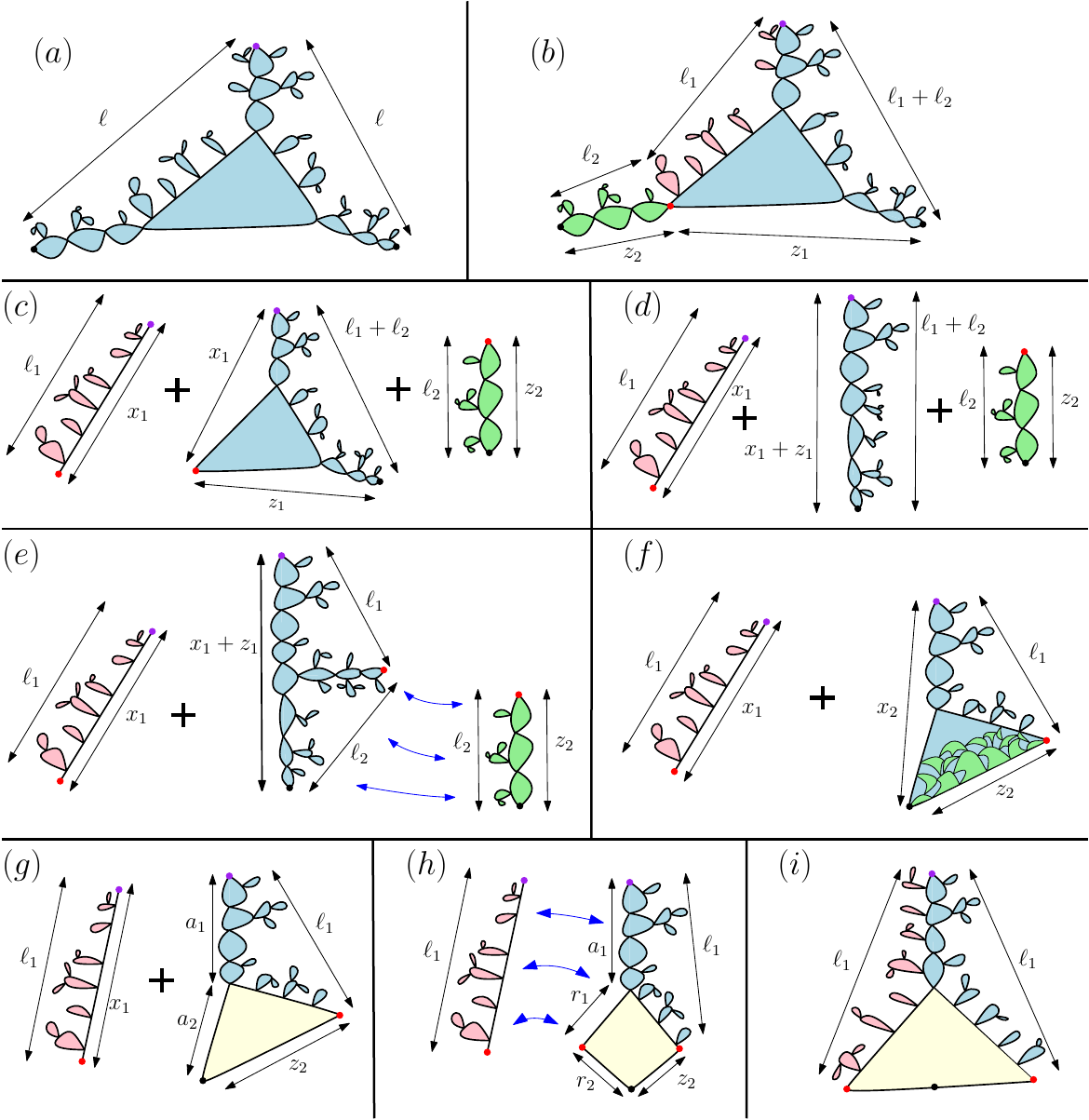}
    \caption{
        Diagram for the proof of Proposition~\ref{prop-weld-last-hit-raSLE}. 
        \textbf{(a) and (b)} represent the left and the right sides of~\eqref{eq-proof-weld-last-hit-raSLE-1}, respectively. 
        The event $A$ corresponds to that the red point in (b) is the last hitting point $\eta(\sigma)$.
        \textbf{(d)} represent the welding in~\eqref{eq-proof-weld-last-hit-raSLE-2}.
        \textbf{From (b) to (d)}, we distinguish the pink forested line segment (see~\textbf{(c)}) and forget the red marked point on the light blue surface. 
        \textbf{(f)} represent the welding in~\eqref{eq-proof-weld-last-hit-raSLE-3}.
        \textbf{From (d) to (f)}, we first add the red marked point on the forested boundary arc of the light blue surface according to generalized quantum length (see~\textbf{(e)}), weld the light blue surface and green surface together and integrate over $\ell_2$, and forget the welding interface and set $x_2:=x_1+z_1$. 
        \textbf{(g) and (h)} represent~\eqref{eq-proof-weld-last-hit-raSLE-4} and \eqref{eq-proof-weld-last-hit-raSLE-5} respectively.
        \textbf{From (f) to (g)}, we set $x_2=a_1+a_2$, and the event $A$ is equivalent to the condition $x_1>a_1$ in (g).
        \textbf{From (g) to (h)}, we separate $a_2$ into $r_1+r_2$ such that $a_1+r_1=x_1$ and mark the red point on the yellow surface.
        \textbf{From (h) to (i)}, we glue the pink forested line segment according to quantum length. 
    }
    \label{fig:prop-weld-last-hit-raSLE}
\end{figure}

\begin{proof}
    See Figure \ref{fig:prop-weld-last-hit-raSLE} for an illustration. The idea is to use~\eqref{eq-weld-raSLE6} in Theorem~\ref{thm-weld-raSLE6}, find the last hitting point $\eta(\sigma)$ on both sides of~\eqref{eq-weld-raSLE6} and then forget the welding interface $\eta(0,\sigma)$. Throughout this proof, the constant $C>0$ can be varying from line to line.
    
    Let $\ell=\ell_1+\ell_2$. Combining Definitions~\ref{def-QT} and~\ref{def-forested-triangle} gives
\begin{equation}
    \label{eq-proof-weld-last-hit-raSLE-1}
    \int_0^\infty\widetilde{\QT}\Big(\frac{2}{3},\frac{2}{3},\frac{2}{3};\ell,\ell\Big)d\ell=\int_{\R_+^4}\cM^{\fd}_{0,2}\Big(\frac{2}{3}\Big)(\ell_2,z_2)\times\widetilde{\QT}\left(2,\frac{2}{3},\frac{2}{3};\ell_1,\ell_1+\ell_2,z_1\right)d\ell_1d\ell_2dz_1dz_2.
\end{equation}
Let $A$ be the event that the common endpoint which concatenates a pair of forested quantum surfaces sampled from the right side of \eqref{eq-proof-weld-last-hit-raSLE-1} is the last hitting point $\eta(\sigma)$. The event $A$ is that the red marked point in (b) of Figure~\ref{fig:prop-weld-last-hit-raSLE} is the last hitting point $\eta(\sigma)$. Equivalently, $A$ happens if the path $(\eta_t)_{0\le t\le \sigma}$ winds counterclockwise, and its complement $A^c$ happens if the path $(\eta_t)_{0\le t\le \sigma}$ winds clockwise. Due to the reflection symmetry of radial SLE$_6$, the measure ${\bf{1}}_A\raSLE_6$ is the pushforward of ${\bf{1}}_{A^c}\raSLE_6$ under the complex conjugate $z\mapsto\bar{z}$.
Restricted on $A$, by Lemma \ref{lem-add-point}, distinguishing all the forested line segments from the forested quantum triangle, and forgetting the marked point of quantum triangle with weight 2, the right side of~\eqref{eq-proof-weld-last-hit-raSLE-1} then becomes
\begin{equation}
\label{eq-proof-weld-last-hit-raSLE-2}
    C\int_{\R_+^5}1_A~\cM^{\fd}_{0,2}\Big(\frac{2}{3}\Big)(\ell_2,z_2)\times \cM^{\fl}_2(x_1,\ell_1)\times\cM^{\fd}_{0,2}\Big(\frac{2}{3}\Big)(x_1+z_1,\ell_1+\ell_2)
    d\ell_1d\ell_2dx_1dz_1dz_2.
\end{equation}
Due to Lemma \ref{lem-f.d.=f.q}, when adding a marked point on the forested boundary arc of a sample from $\cM^{\fd}_{0,2}\Big(\frac{2}{3}\Big)(x_1+z_1,\ell_1+\ell_2)$ (which separate this arc into two arcs with length $\ell_1,\ell_2$),the output triply marked surface has the law $\wt{\QT}(\frac{2}{3},\frac{2}{3},\frac{2}{3};\ell_1,\ell_2,x_1+z_1)$.
Then according to Lemma~\ref{lem-weld-two-f.d.}, we know that by conformally welding $\cM^{\fd}_{0,2}\Big(\frac{2}{3}\Big)(\ell_2,z_2)$ and $\wt{\QT}(\frac{2}{3},\frac{2}{3},\frac{2}{3};\ell_1,\ell_2,x_1+z_1)$, integrating over $\ell_2$ and finally forgetting the welding interface (this corresponds to forget $\eta(0,\sigma)$), \eqref{eq-proof-weld-last-hit-raSLE-2} becomes
\begin{equation}
    \label{eq-proof-weld-last-hit-raSLE-3}
    C\int_{\R_+^5}1_A1_{x_2>x_1}~\cM^{\fl}_2(x_1,\ell_1)\times\QT\Big(2,\frac{2}{3},2\Big)(x_2,y,z_2)\times\cM^{\fl}_2(y,\ell_1)d\ell_1dx_1dx_2dz_2dy,
\end{equation}
where we set $x_2:=x_1+z_1$.
According to Definition~\ref{def-QT}, by setting $y=y_1+y_2,x_2=a_1+a_2$, we can further decompose~\eqref{eq-proof-weld-last-hit-raSLE-3} into
\begin{equation}
    \label{eq-proof-weld-last-hit-raSLE-4}
    \begin{aligned}
        C\int_{\R_+^7}1_A1_{a_1+a_2>x_1}~\cM^{\fl}_2(x_1,\ell_1)\times
    \cM^{\disk}_{0,2}\Big(\frac{2}{3}\Big)(a_1,y_1)\times
    \QT(2,2,2)(a_2,y_2,z_2)\times\cM^{\fl}_2(y_1+y_2,\ell_1)\\
    d\ell_1dx_1da_1da_2dz_2dy_1dy_2.
    \end{aligned}
\end{equation}
Note that restricted on $A$ is equivalent to requiring $x_1>a_1$. Let $r_1:=x_1-a_1$ and $r_2:=a_2-r_1$. Now \eqref{eq-proof-weld-last-hit-raSLE-4} becomes
\begin{equation}
    \label{eq-proof-weld-last-hit-raSLE-5}
    \begin{aligned}
        C\int_{\R_+^7}\cM^{\fl}_2(a_1+r_1,\ell_1)\times
    \cM^{\disk}_{0,2}\Big(\frac{2}{3}\Big)(a_1,y_1)\times
    \QD_{0,4}(r_1,r_2,z_2,y_2)\times\cM^{\fl}_2(y_1+y_2,\ell_1)\\
    d\ell_1da_1dr_1dr_2dz_2dy_1dy_2,
    \end{aligned}
\end{equation}
which indeed equals $C\int_0^\infty\widetilde{\QT}_\bullet(2,\frac{2}{3},2;\ell_1,\ell_1)d\ell_1$ (see Definitions \ref{def-QT} and \ref{def-forested-triangle}). In conclusion, when restricted on the event $A$, the law of the quantum surface obtained by $\int_0^\infty\Weld(\widetilde{\QT}(\frac{2}{3},\frac{2}{3},\frac{2}{3};\ell,\ell))d\ell$ and forgetting the welding interface before the last hitting point  is equal to $C\int_0^\infty\Weld(\widetilde{\QT}_\bullet(2,\frac{2}{3},2;\ell,\ell))d\ell$.
This  also holds when not restricted on $A$ due to the symmetry of $A$ and $A^c$ explained above.
On the other hand, $\QD_{1,1}\otimes\raSLE_6^\sigma$ is equal to $\QD_{1,1}\otimes\raSLE_6$ after forgetting $\eta(0,\sigma)$. Therefore, by Theorem \ref{thm-weld-raSLE6}, we conclude the proof.
\end{proof}

Let $(\D,\phi,0)$ be an embedding of a sample from $\QD_{1,0}$. Sample $\eta$ from radial $\SLE_6$ in $\D$ from 0 to 1 independently. Let $\sigma$ be the last time that $\eta$ hits $\partial\D$. We write $\QD_{1,0}\otimes\raSLE_6^\sigma$ for the law of curve decorated quantum surface $(\D,\phi,\eta[\sigma,\infty),0)/\mathord\sim_\gamma$. The following is a quick consequence of Proposition~\ref{prop-weld-last-hit-raSLE}.
\begin{corollary}
     \label{prop-weld-last-hit-raSLE-2}
    There exists a constant $C>0$ such that
    \begin{equation*}
        \QD_{1,0}\otimes\raSLE_6^\sigma=C\int_0^\infty\Weld\left(\widetilde{\QT}\left(2,\frac{2}{3},2;\ell,\ell\right)\right)d\ell.
    \end{equation*}
\end{corollary}
\begin{proof}
    By Theorem \ref{thm-raSLE6-last-hitting}, the law of $\eta[\sigma,\infty)$ does not rely on the position of the starting point of $\eta$. Then the result follows from Proposition~\ref{prop-weld-last-hit-raSLE} by forgetting the boundary marked point.
\end{proof}

The following proposition gives a conformal welding description of $\int_0^\infty\Weld\left(\widetilde{\QT}\left(2,\frac{2}{3},2;\ell,\ell\right)\right)d\ell$ obtained in Corollary~\ref{prop-weld-last-hit-raSLE-2}.

\begin{proposition}
\label{prop-weld-QT=QD}
Consider the quantum surface sampled from $\int_0^\infty \Weld\left(\widetilde{\QT}\left(2,\frac{2}{3},2;\ell,\ell\right)\right)d\ell$. If we only keep the outer boundary of the welding interface (and forget the remaining part of the welding interface), then there exists a constant $C>0$ such that the resulting curve decorated quantum surface has the same law as
\begin{equation}
    \label{eq-weld-QT=QD}
    \begin{aligned}
        C   \int_{\R_+^4}\Weld\left(\QD_{1,1}(b),\QD_{0,4}(a,b,c,L),\cM^{\disk}_{0,2}\left(\frac23\right)(a,c)\right)dadbdcdL.
    \end{aligned}
\end{equation}

\end{proposition}

Note that Theorem~\ref{thm:weld-BM1} quickly follows by combining Corollary~\ref{prop-weld-last-hit-raSLE-2} and Proposition~\ref{prop-weld-QT=QD}.
\begin{proof}[Proof of Theorem \ref{thm:weld-BM1} given Proposition \ref{prop-weld-QT=QD}.]
    By the equivalence between Brownian motion and radial $\SLE_6$ in Lemma~\ref{lem-BM-raSLE6}, if we only keep the outer boundary (and forget the remaining part of the curve) in a sample from $\QD_{1,0}\otimes\raSLE_6^\sigma$, the resulting curve decorated quantum surface has the law $\QD_{1,0}\otimes\mathsf{P}$. Then Theorem \ref{thm:weld-BM1} follows from Corollary~\ref{prop-weld-last-hit-raSLE-2}  and  Proposition~\ref{prop-weld-QT=QD}.
\end{proof}

The remaining of this section is devoted to proving Proposition \ref{prop-weld-QT=QD}. We first need the following lemma on conformal welding of a quantum triangle and a quantum disk. For $b>0$, define $\QT(2,\frac{2}{3},2;b)=\int_0^\infty\QT(2,\frac{2}{3},2;b,b')db'$.
We write $\int_0^\infty\Weld(\QD_{1,1}(b),\QT(2,\frac{2}{3},2;b))db$ for the law of the curve decorated quantum surface obtained by the conformal welding of a pair of quantum surfaces sampled from $\int_0^\infty\QD_{1,1}(b)\times\QT(2,\frac{2}{3},2;b)db$.

\begin{lemma}
        \label{lem-weld-QT-surround-QD}
Consider a sample from $\int_0^\infty\Weld(\QD_{1,1}(b),\QT(2,\frac{2}{3},2;b))db$. If we forget the welding interface and the boundary marked point of $\QD_{1,1}(b)$, then the output quantum surface has the law of a constant multiple of $\QD_{1,1}$. See (d) in Figure \ref{fig:weld-QT=QD}.
\end{lemma}

\begin{proof}
    By \cite[Lemma 4.7]{SXZ24}, if we concatenate two vertices of a sample from  $\QT(2,\frac{2}{3},2)$ with weight $2$ and $\frac{2}{3}$ together, then we obtain a pinched quantum annulus with one marked point, whose law is called $\widetilde{\QA}_{\bullet}(\frac{2}{3})$ there. Then we have
    \begin{equation}\label{eq-proof-weld-QT-surround-QD}
        \int_0^\infty \Weld\left(\QD_{1,1}(b),\QT\left(2,\frac{2}{3},2;b\right)\right)db=\int_0^\infty b\Weld\left(\QD_{1,0}(b),\widetilde{\QA}_{\bullet}\left(\frac{2}{3}\right)(b)\right)db
    \end{equation}
where the right side of \eqref{eq-proof-weld-QT-surround-QD} is the uniform conformal welding. By~\cite[Proposition 4.23]{SXZ24}, the right side of \eqref{eq-proof-weld-QT-surround-QD} after forgetting the welding interface is a constant multiple of $\QD_{1,1}$.
\end{proof}

Fix $\ell_1,\ell_2>0$. Let $(\mathcal{L},\widetilde{\mathcal{D}})$ be sampled from $\int_0^\infty\cM^{\fl}_2(t,\ell_1)\times\cM^{\fd}_{0,2}(\frac{2}{3})(\ell_2,t)dt$, and glue the line boundary arc of $\mathcal{L}$ to the unforested boundary arc of $\widetilde{\mathcal{D}}$ with length $t$
according to the quantum length. We denote the law of the resulting forested quantum surface by $\cM^{\fdd}_{0,2}(\frac{2}{3})(\ell_1,\ell_2)$.

\begin{figure}[!htbp]
    \centering
    \includegraphics[width=1\linewidth]{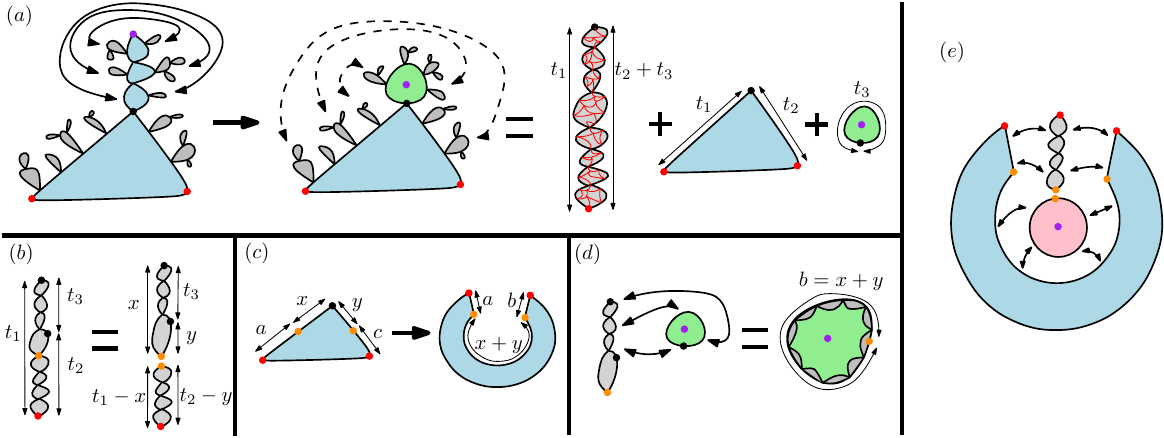}
    \caption{
    Diagram for the proof of Proposition~\ref{prop-weld-QT=QD}. 
        \textbf{(a)} \textbf{Left:} A sample from $\widetilde{\QT}\left(2,\frac{2}{3},2;\ell,\ell\right)$, whose law is given by~\eqref{eq-proof-weld-QT=QD-1}. 
        By \cite[Lemma 4.1]{ASYZ24} and Theorem \ref{thm-weld-raSLE6}, welding $\cM^{\fdd}_{0,2}\left(\frac{2}{3}\right)(\ell_1,\ell_1+\ell_3)$ and forgetting the welding interface gives the middle panel. 
        \textbf{Middle:} Conformally weld the two forested boundary arcs. 
        \textbf{Right:} The joint law of the three quantum surfaces is given by~\eqref{eq-proof-weld-QT=QD-6}. 
        \textbf{(b)} Mark the point on $\cM^{\disk}_{0,2}\left(\frac{2}{3}\right)(t_3+t_2,t_1)$ such that the quantum length on the right boundary between the two black marked points is $t_3$, which yields a quantum triangle $\QT(2,\frac{2}{3},\frac{2}{3};t_3,t_1,t_2)$. Then decompose it using~\eqref{eq-proof-weld-QT=QD-*}.  
        \textbf{(c)} Mark two points in $\QD_{0,3}(x+a,y+c,\cdot)$ and forget the black point.
        \textbf{(d)} Apply Lemma \ref{lem-weld-QT-surround-QD} to conformally weld $\QD_{1,1}(t_3)$ and $\QT(2,\frac{2}{3},2;t_3,x,b-x)$ from (\ref{eq-proof-weld-QT=QD-7}), and then forget the welding interface, which gives $\QD_{1,1}(b)$.
        \textbf{(e)} Combine the three components obtained in (b), (c), and (d), whose joint law is given by (\ref{eq-weld-QT=QD}).
    }
    \label{fig:weld-QT=QD}
\end{figure}

\begin{proof}[Proof of Proposition \ref{prop-weld-QT=QD}]
See Figure~\ref{fig:weld-QT=QD} for an illustration.
The constant $C$ appears in the following can vary from line to line. By Definitions~\ref{def-forested-triangle} and~\ref{def-QT} and change of variables, we have
\begin{equation}
\label{eq-proof-weld-QT=QD-1}
    \int_0^\infty \widetilde{\QT}\left(2,\frac{2}{3},2;\ell,\ell\right)d\ell=\int_{\R_+^3}1_{\ell>\ell_1,\ell>\ell_2}\widetilde{\QT}(2,2,2;\ell-\ell_1,\ell-\ell_2)\times\cM^{\fdd}_{0,2}\left(\frac{2}{3}\right)(\ell_1,\ell_2)d\ell_1d\ell_2d\ell.
\end{equation}

By symmetry, we can assume $\ell_2>\ell_1$ without loss of generality. Set $\ell_3:=\ell_2-\ell_1,\ell_4:=\ell-\ell_2$, and then the right side of \eqref{eq-proof-weld-QT=QD-1} becomes
\begin{equation}
\label{eq-proof-weld-QT=QD-2}
    \int_{\R_+^3}\widetilde{\QT}(2,2,2;\ell_3+\ell_4,\ell_4)\times\cM^{\fdd}_{0,2}\left(\frac{2}{3}\right)(\ell_1,\ell_1+\ell_3)d\ell_1d\ell_3d\ell_4.
\end{equation}

By \cite[Lemma 4.1]{ASYZ24}, if we mark the point on the forested boundary arc with length $\ell_1+\ell_3$ of a sample from $\cM^{\fdd}_{0,2}\left(\frac{2}{3}\right)(\ell_1,\ell_1+\ell_3)$, which separates the arc into two forested arcs with length $\ell_1$ and $\ell_3$, then we obtain $\int_0^\infty\widetilde{\QT}(\frac{2}{3},\frac{2}{3},\frac{2}{3};\ell_1,\ell_1,t_3)\times\cM^{\fl}_2(t_3,\ell_3)dt_3$. By Theorem \ref{thm-weld-raSLE6}, if we weld a sample from $\widetilde{\QT}(\frac{2}{3},\frac{2}{3},\frac{2}{3};\ell_1,\ell_1,t_3)$, integrate over $\ell_1$, and forget the welding interface, then the resulting quantum surface has the law $C\QD_{1,1}(t_3)$. Then \eqref{eq-proof-weld-QT=QD-2} becomes
\begin{equation}
\label{eq-proof-weld-QT=QD-3}
    C\int_{\R_+^3}\widetilde{\QT}(2,2,2;\ell_3+\ell_4,\ell_4)\times\QD_{1,1}(t_3)\times\cM^{\fl}_2(t_3,\ell_3)
    dt_3d\ell_3d\ell_4.
\end{equation}
Then by distinguishing forested line segments from the forested quantum triangle in \eqref{eq-proof-weld-QT=QD-3}, we have
\begin{equation}
    \label{eq-proof-weld-QT=QD-4}
C\int_{\R_+^5}\QT(2,2,2;t_1,t_2)\times\QD_{1,1}(t_3)\times\cM^{\fl}_2(t_1,\ell_3+\ell_4)\times\cM^{\fl}_2(t_2,\ell_4)\times\cM^{\fl}_2(t_3,\ell_3)dt_1dt_2dt_3d\ell_3d\ell_4.
\end{equation}
Fix $t_2,t_3$ and $\ell':=\ell_3+\ell_4$, concatenate a sample from $\cM^{\fl}_2(t_2,\ell_4)$ and a sample from  $\cM^{\fl}_2(t_3,\ell'-\ell_4)$, 
\begin{equation*}
    \int_{0}^{\ell'}\cM^{\fl}_2(t_2,\ell_4)\times \cM^{\fl}_2(t_3,\ell'-\ell_4)d\ell_4=\cM^{\fl}_2(t_3+t_2,\ell').
\end{equation*}
Then \eqref{eq-proof-weld-QT=QD-4} becomes
\begin{equation}
    \label{eq-proof-weld-QT=QD-5}
C\int_{\R_+^4}\QT(2,2,2;t_1,t_2)\times\QD_{1,1}(t_3)\times\cM^{\fl}_2(t_1,\ell')\times\cM^{\fl}_2(t_3+t_2,\ell')dt_1dt_2dt_3d\ell'.
\end{equation}
By Proposition \ref{thm-weld-f.l.s}, we 
weld two forested line segments in \eqref{eq-proof-weld-QT=QD-5} according to generalized quantum length, integrate over $\ell'$, and forget the welding interface, then \eqref{eq-proof-weld-QT=QD-5} becomes
\begin{equation}
    \label{eq-proof-weld-QT=QD-6}
    C\int_{\R_+^3}\QT(2,2,2;t_1,t_2)\times\QD_{1,1}(t_3)\times\cM^{\disk}_{0,2}\left(\frac{2}{3}\right)(t_3+t_2,t_1)dt_1dt_2dt_3.
\end{equation}
See (a) in Figure~\ref{fig:weld-QT=QD}.

By Lemma \ref{lem-add-point}, if we mark the point on the boundary arc with length $t_3+t_2$ of a sample from $\cM^{\disk}_{0,2}(\frac{2}{3})(t_3+t_2,t_1)$ which separate this arc into two boundary arcs with length $t_3$ and $t_2$, then the resulting quantum surface has the law $\QT(2,\frac{2}{3},\frac{2}{3};t_3,t_1,t_2)$. By Definition~\ref{def-QT}, we have
\begin{equation}
    \label{eq-proof-weld-QT=QD-*}
    \QT\left(2,\frac{2}{3},\frac{2}{3};t_3,t_1,t_2\right)=\int_0^{t_1}\int_{0}^{t_2}\QT\left(2,\frac{2}{3},2;t_3,x,y\right)\times\cM^{\disk}_{0,2}\left(\frac{2}{3}\right)(t_2-y,t_1-x)dydx.
\end{equation}
See (b) in Figure~\ref{fig:weld-QT=QD}.
Further set $a:=t_1-x,c:=t_2-y$ and substitute the decomposition \eqref{eq-proof-weld-QT=QD-*} into \eqref{eq-proof-weld-QT=QD-6}, then~\eqref{eq-proof-weld-QT=QD-6} becomes
\begin{equation}
    \label{eq-proof-weld-QT=QD-7}
    C\int_{\R_+^5}\QT(2,2,2;x+a,y+c)\times\QD_{1,1}(t_3)\times\QT\left(2,\frac{2}{3},2;t_3,x,y\right)\times\cM^{\disk}_{0,2}\left(\frac{2}{3}\right)(c,a)dxdydadcdt_3.
\end{equation}

Note that from Definition \ref{def-QT} we have $\QT(2,2,2;x+a,y+c)=C\QD_{0,3}(x+a,y+c,\cdot)$. For a sample from $\QD_{0,3}(x+a,y+c,\cdot)$, we mark two points on the two boundary arcs with length $x+a$ and $y+c$ respectively, which separate the two boundary arcs into four boundary arcs with length $a,x,y,c$ in clockwise order, then the output quantum surface with five marked points has the law $\QD_{0,5}(a,x,y,c,\cdot)$. Furthermore, for a sample from $\QD_{0,5}(a,x,y,c,\cdot)$, we forget the marked point which is the common endpoint of two boundary arcs with length $x$ and $y$, then the resulting quantum surface with four marked point has the law $\QD_{0,4}(a,x+y,c,\cdot)$. See (c) in Figure~\ref{fig:weld-QT=QD}.

Now for $b>0$, if we forget the welding interface of $\int_0^b\int_0^\infty\Weld(\QD_{1,1}(t_3),\QT(2,\frac{2}{3},2;t_3,x,b-x))dt_3dx$, then the resulting quantum surface has the law $C\QD_{1,1}(b)$ by Lemma \ref{lem-weld-QT-surround-QD}.
Therefore,~\eqref{eq-proof-weld-QT=QD-7} becomes (see (e) in Figure~\ref{fig:weld-QT=QD})
\begin{equation}
    \label{eq-proof-weld-QT=QD-8}
    C\int_{\R_+^4}\QD_{0,4}(a,b,c,L)\times\QD_{1,1}(b)\times\cM^{\disk}_{0,2}\left(\frac{2}{3}\right)(c,a)dadbdcdL.
\end{equation}

In conclusion, we find that the quantum surface from $\int_0^\infty \widetilde{\QT}\left(2,\frac{2}{3},2;\ell,\ell\right)d\ell$ (only keeping the outermost welding interfaces) in the left side of~\eqref{eq-proof-weld-QT=QD-1} is equal to the quantum surface from~\eqref{eq-proof-weld-QT=QD-8}. The proposition then follows.
\end{proof}

\section{Proof of Theorem~\ref{thm:sm}}
\label{section:proof of thm1.2}

In this section we prove Theorem~\ref{thm:sm} based on the conformal welding result in Theorem~\ref{thm:weld-BM1}. We fix $\gamma=\sqrt{8/3}$ throughout this section as before.
Recall that for a Brownian motion $(B_t)_{0\le t\le\tau_\D}$ in $\D$ from $0$ to the boundary, $\ell$ is the outer boundary of $(B_t)_{0\le t\le\tau_\D}$, and $\sm$ is the law of the boundary $\eta$ of the connected component containing $0$ of $\D\setminus\ell$. Let $A_\eta$ and $D_\eta$ be the two connected components of $\D\backslash\eta$ with annular and disk topology, respectively.

For $L>0$, let $(\D,\phi,0)$ be an embedding of a sample from $\QD_{1,0}(L)$. Sample $\eta$ from $\sm$ independently. We denote the law of $(\D,\phi,0,\eta)/\mathord\sim_\gamma$ by $\QD_{1,0}(L)\otimes \sm$ and the total quantum length of $\eta$ under $\phi$ by $\nu_\phi(\eta)$. According to Theorem \ref{thm:weld-BM1}, the quantum length measure of $(A_\eta,\phi)/\mathord\sim_\gamma$ and $(D_\eta,\phi)/\mathord\sim_\gamma$ agree on $\eta$, then $\nu_\phi(\eta)$ is well-defined.

We now introduce the key  quantum surface $\RA$ in this section.
Note that for any Borel set $E\subset\R$ with zero Lebesgue measure, $\QD_{1,0}(L)\otimes\sm[\nu_\phi(\eta)\in E]=0$ by Theorem \ref{thm:weld-BM1}. Hence we can define the disintegration of $\QD_{1,0}(L)\otimes\sm$ over $\nu_\phi(\eta)$, which we denote by $\{\QD_{1,0}(L)\otimes\sm(b),b\in(0,\infty)\}$.

\begin{definition}\label{def:RA}
For $b,L>0$, suppose $(\D,\phi,0,\eta)/\mathord\sim_\gamma$ is sampled from $\QD_{1,0}(L)\otimes\sm(b)$. Let $\RA'(b,L)$ be the law of the quantum surface $(A_\eta,\phi)/\mathord\sim_\gamma$. Let $\RA(b,L)$ be such that
$$b|\QD_{1,0}(b)|\RA(b,L)=\RA'(b,L),$$
and let $\RA=\int_{\R_+^2}\RA(b,L)dbdL$.
\end{definition}

In this section we use $\RA$ as a  tool to prove Theorem~\ref{thm:sm}.
In Section~\ref{section-weld-matter-BA} we will relate $\RA$ to  the Brownian annulus decorated with a non-disconnecting Brownian excursion.
We need the following conformal welding results of $\RA$.

\begin{proposition}
    \label{prop-weld-RA}
For $L>0$, we have
\begin{equation*}
    \QD_{1,0}(L)\otimes \sm=\int_0^\infty b\Weld(\QD_{1,0}(b),\RA(b,L))db.
\end{equation*}
    Here $\Weld(\QD_{1,0}(b),\RA(b,L))$ denotes the uniform conformal welding of a pair of quantum surfaces sampled from $\QD_{1,0}(b)\times\RA(b,L)$.
\end{proposition}
\begin{proof}

    Recall the setup in Theorem~\ref{thm:weld-BM1}. Let $(\D,\phi,0)$ be an embedding of a sample from $\QD_{1,0}(L)$, and sample $\ell$ from $\mathsf{P}$ independently. Let $\eta\subset\ell$ be the Jordan loop surrounding 0 (thus the decorated quantum surface $(\D,\phi,0,\eta)/\mathord\sim_\gamma$ has the law $\QD_{1,0}(L)\otimes\sm$). Then according to Theorem~\ref{thm:weld-BM1} and Definition~\ref{def:RA}, the pair of quantum surfaces $(D_\eta,\phi,0)/\mathord\sim_\gamma, (A_\eta,\phi)/\mathord\sim_\gamma$ has the law $\int_0^\infty b\QD_{1,0}(b)\times\RA(b,L)db$, where $b=\nu_\phi(\eta)$. It then suffices to show the welding is uniform.

     Let $p$ be the cut point of $\ell$ on $\eta$. By Theorem \ref{thm:weld-BM1}, 
    conditioned on $(A_\eta,\phi,p)/\mathord\sim_\gamma$, the conditional law of $(D_\eta,\phi,0,p)/\mathord\sim_\gamma$ is $\QD_{1,1}(\nu_\phi(\eta))^{\#}$. 
Let $U$ be an independent uniform random variable on $(0,1)$ and $\omega\in\eta$ such that the counterclockwise arc on $\eta$ from $p$ to $\omega$ has quantum length $U\nu_\phi(\eta)$. 
The re-rooting invariance of  $\QD_{1,1}(\nu_\phi(\eta))^{\#}$ then implies that given $(A_\eta,\phi,p)/\mathord\sim_\gamma$ and $U$, the conditional law of $(D_\eta,\phi,0,\omega)/\mathord\sim_\gamma$ is still $\QD_{1,1}(\nu_\phi(\eta))^{\#}$. Since $(A_\eta,\phi,\omega)/\mathord\sim_\gamma$ is determined by  $(A_\eta,\phi,p)/\mathord\sim_\gamma$ and $U$, we conclude that conditioned on $\nu_\phi(\eta)$, $(A_\eta,\phi,\omega)/\mathord\sim_\gamma$ and $(D_\eta,\phi,0,\omega)/\mathord\sim_\gamma$ is conditionally independent, and the conditional law of $(D_\eta,\phi,0,\omega)/\mathord\sim_\gamma$ is $\QD_{1,1}(\nu_\phi(\eta))^{\#}$. The result then follows by the definition of uniform conformal welding (see Section~\ref{sec-conformal-welding}).
\end{proof}

In order to prove Theorem~\ref{prop: key lemma}, the main ingredient is the following Liouville field description of $\mathcal{BA}$.
For $\tau>0$, let $\mathcal{C}_\tau=[0,\tau]\times[0,1]/\mathord\sim$ be a finite horizontal cylinder with modulus $\tau$, where $\sim$ means we identify $(x,0)$ and $(x,1)$ for each $x\in[0,\tau]$.  
Let $\mathbb{P}_{\tau,\rho}$ denote the law of the free boundary GFF on $\mathcal{C}_\tau$ constructed in Section~\ref{sec:LF}. For $(h,\textbf{c})$ sampled from $\mathbb{P}_{\tau,\rho}\times dc$, define the Liouville field $\phi$ on $\mathcal{C}_\tau$ to be $\phi:=h+\textbf{c}$, and we denote the law of $\phi$ by $\LF_\tau$. Note that the measure $\LF_\tau$ does not depend on the choice of $\rho$ due to the translation invariance of the Lebesgue measure $dc$.

\begin{proposition}\label{prop:LF-description}
Recall $f(\tau)=\sum_{n\ge1}(-1)^{n-1}n\sin(\frac{2\pi}{3} n)\exp(-\frac{2\pi}{3}n^2\tau)$ in~\eqref{eq:f}. Let $(\phi,\tau)$ be sampled from from ${\bf 1}_{\tau>0} \LF_\tau(d\phi) f(\tau)d\tau$. Then the law of $(\mathcal{C}_\tau,\phi)/\mathord\sim_\gamma$ equals a constant multiple of $\mathcal{BA}$.
\end{proposition}

In the remaining part of this section, we first finish the proof of Theorem~\ref{thm:sm} in Section~\ref{sec:proof-of-1.2} based on Proposition~\ref{prop:LF-description}, and then prove Proposition~\ref{prop:LF-description} in Section~\ref{sec:proof-of-4.3}. The proof of Proposition~\ref{prop:LF-description} has two steps.
The first step is to show that there exists some measure $m(d\tau)$ on $\R_+$ such that $\RA$ is equal to $\LF_\tau(d\phi)m(d\tau)$. This will follow from the symmetry of the two boundaries of $\RA$ and a standard argument in~\cite[Section 4.2]{ARS2022moduli}. The second step is to identify the expression of $m(d\tau)$, where we will use the conformal welding description of $\RA$ from Theorem~\ref{thm:weld-BM1} as well as the integrability of $\LF_\tau$ from~\cite{Wu22}.

\subsection{Proof of Theorem~\ref{thm:sm} given Proposition~\ref{prop:LF-description}}\label{sec:proof-of-1.2}

We first recall the conformal welding arising from the $\SLE_{8/3}$ loop decorated Brownian disk~\cite{ARS2022moduli}. 

\begin{definition}[{\cite{ARS2022moduli}}]\label{def:BA}
    Let $\gamma=\sqrt{8/3}$ and sample $(\phi,\tau)$ from ${\bf{1}}_{\tau>0}\LF_\tau(d\phi)\eta(2i\tau)d\tau$. Define $\BA$ to be  the law of quantum surface $(\mathcal{C}_\tau,\phi)/\mathord\sim_\gamma$.
    We also write $(\BA(b,L))_{b,L>0}$ to be the disintegration of $\BA$ under two boundary quantum lengths, i.e. $\BA=\iint_{\R_+^2}\BA(b,L)dbdL$.
\end{definition}

For $L>0$, let $(\D,\phi,0)$ be an embedding of a sample from $\QD_{1,0}(L)$.
Recall that $\SLE_{8/3,\D}^{\lp}$ is the $\SLE_{8/3}$ loop measure restricted to the loops contained in $\D$ and surrounding 0.
Sample $\eta$ independently from $\SLE_{8/3,\D}^{\lp}$, and denote the law of $(\D,\phi,0,\eta)/\mathord\sim_\gamma$ by $\QD_{1,0}(L)\otimes \SLE_{8/3}^{\lp}$.

The following gives the conformal welding description of $\QD_{1,0}(L)\otimes \SLE_{8/3}^{\lp}$, which is implicit in~\cite{ARS2022moduli}. For completeness,
we provide a proof in Appendix~\ref{sec:loop-app}.

\begin{proposition}\label{thm:weld-BA}
There exists a constant $C>0$, such that for each $L>0$, we have
\begin{equation*}
        \QD_{1,0}(L)\otimes\SLE_{8/3}^{\lp}=C\int_0^\infty b\Weld(\QD_{1,0}(b),\BA(b,L))db.
    \end{equation*}
    
\end{proposition}

The proof of Theorem~\ref{prop: key lemma} is then based on comparing the conformal welding descriptions of $\sm(d\eta)$ and $\SLE_{8/3,\D}^\lp(d\eta)$. Namely, according to Proposition~\ref{prop-weld-RA} (resp.~Proposition~\ref{thm:weld-BA}), we see that $\sm(d\eta)$ (resp.~$\SLE_{8/3,\D}^\lp(d\eta)$) can be realized as the welding interface of a quantum disk and $\RA$ (resp.~$\BA$). Furthermore, by Proposition~\ref{prop:LF-description} and Definition~\ref{def:BA}, we know that $\RA$ is mutually absolutely continuous with respect to $\BA$ (and their Radon-Nikodym derivative is exactly given by the ratio of moduli densities, i.e.~$f(\tau)/\eta(2i\tau)$). This then readily implies Theorem~\ref{prop: key lemma}.

\begin{proof}[Proof of Theorem~\ref{prop: key lemma}]
For $L>0$, let $\cM(L)$ be the collection of the pairs $(\cS_1,\cS_2)$ of $\sqrt{8/3}$-quantum surfaces, such that $\cS_1$ is of disk topology with one bulk marked point, $\cS_2$ is of annular topology, and the quantum length of the boundary $\partial\cS_1$ of $\cS_1$ coincides with one of the boundaries $\partial^{\rm in}\cS_2$ of $\cS_2$, while the other boundary $\partial^{\rm out}\cS_2$ of $\cS_2$ has quantum length $L$. For $(\cS_1,\cS_2)\in\mathcal{M}(L)$, we can uniformly weld $\partial\cS_1$ to $\partial^{\rm in}\cS_2$ to obtain a loop decorated quantum surface $\cS$ with one bulk marked point $\mathbf{p}$, and uniformly embedding\footnote{The uniform embedding here means that after requiring the embedding such that $\mathbf{p}$ is embedded to $0$, we choose the remaining degree of freedom (i.e.~rotation) to be uniform on $\S^1$.} $\cS$ onto $\D$ such that $\mathbf{p}$ is embedded to $0$ will output a simple loop $\eta\subset\D$ surrounding $0$. This defines a map $F$ from $\cM(L)\times\S^1\times\S^1$ to the collection of simple loops in $\D$ surrounding $0$; here we identify the randomnesses of uniform welding or embedding by two independent random variables which uniformly take values on $\S^1$. 

Note that $\mu:=\int_0^\infty b\QD_{1,0}(b)\times\RA(b,L)db$ and $\nu:=C\int_0^\infty b\QD_{1,0}(b)\times\BA(b,L)db$ define two measures on $\cM(L)$ (here the constant $C$ is the same as in Proposition~\ref{thm:weld-BA}). Comparing the Liouville field descriptions of $\BA$ and $\RA$ in Definition~\ref{def:BA} and Proposition~\ref{prop:LF-description} then implies
\begin{equation}\label{eq:rn-uv}
\frac{d(\mu\otimes({\rm Unif}_{\S^1})^2)}{d(\nu\otimes({\rm Unif}_{\S^1})^2)}((\cS_1,\cS_2),\theta_1,\theta_2)=\frac{1}{C}\frac{f(\tau)}{\eta(2i\tau)},\quad\quad\text{ where }\tau\text{ is the modulus of }\cS_2.
\end{equation}
Since the right side of~\eqref{eq:rn-uv} only depends on the modulus of $\cS_2$, for any simple loop $\eta\subset\D$ surrounding $0$, the above Radon-Nikodym derivative~\eqref{eq:rn-uv} equals $\frac{f(\Mod(A_\eta))}{\eta(2i\Mod(A_\eta))}$ on its pre-image set $F^{-1}(\eta)$ of the map $F$, hence is a constant on $F^{-1}(\eta)$. Therefore, for the push-forward measures $F_*(\mu\otimes({\rm Unif}_{\S^1})^2)$ and $F_*(\nu\otimes({\rm Unif}_{\S^1})^2)$ on the space of simple loops in $\D$, we have
\begin{equation}\label{eq:rn-uv-2}
\frac{dF_*(\mu\otimes({\rm Unif}_{\S^1})^2)}{dF_*(\nu\otimes({\rm Unif}_{\S^1})^2)}(\eta)=\frac{1}{C}\frac{f(\tau)}{\eta(2i\tau)},\quad\quad\text{ where }\tau=\Mod(A_\eta).
\end{equation}

On the other hand, according to Proposition~\ref{prop-weld-RA} (resp.~Proposition~\ref{thm:weld-BA}), for $((\cS_1,\cS_2),\theta_1,\theta_2)$ sampled from $\mu\otimes({\rm Unif}_{\S^1})^2$ (resp.~$\nu\otimes({\rm Unif}_{\S^1})^2$), the law of the output loop $F((\cS_1,\cS_2),\theta_1,\theta_2)$ is given by $\sm(d\eta)$ (resp.~$\SLE_{8/3,\D}^\lp$). Hence,~\eqref{eq:rn-uv-2} implies that $\frac{d\sm}{d\SLE_{8/3,\D}^{\lp}}(\eta)=\frac{1}{C}\frac{f(\tau)}{\eta(2i\tau)}$ for $\tau=\Mod(A_\eta)$, which concludes the proof.
\end{proof}

\subsection{The Liouville field identification of $\RA$}\label{sec:proof-of-4.3}

This section is devoted to prove Proposition \ref{prop:LF-description}.
We first show that given the modulus $\tau$, the field of $\RA$ is indeed described by $\LF_\tau$.
\begin{proposition}
    \label{prop-LF-identify-RA}
    There exists a positive Borel measure $m(d\tau)$ on $\R_+$, such that if we sample $(\phi,\tau)$ from $\LF_\tau(d\phi)m(d\tau)$, the quantum surface $(\mathcal{C}_\tau,\phi)/\mathord\sim_\gamma$ has the same law as $\RA$.
\end{proposition}

We will follow the approach in~\cite[Section 4.2]{ARS2022moduli} to prove Proposition~\ref{prop-LF-identify-RA}.
The key ingredient is the following symmetry of the inner and outer boundaries of $\RA$.
\begin{proposition}
\label{prop-symmetry-RA}
For each $b,L>0$, we have $\RA(b,L)=\RA(L,b)$.
\end{proposition}
\begin{proof}
Recall that $\SLE_{8/3}^{\sep}$ is the $\SLE_{8/3}$ loop measure on $\widehat{\C}$ restricted to separate 0 and $\infty$. Let $\eta$ be sampled from $\SLE_{8/3}^{\sep}$. Given $\eta$, sample an independent Brownian motion $(B_t)_{0\le t\le \tau_\D}$ on $\widehat{\C}$ starting from 0 until first hitting $\eta$. Let $\ell$ be the outer boundary of $(B_t)_{0\le t\le \tau_\D}$. Denote the law of $\eta\cup\ell$ to be $\mathsf{n}$. Let $(\widehat{\C},\phi,0,\infty)$ be an embedding of an independent sample from $\QS_2$. Then by combining Proposition~\ref{prop-weld-loop} and Theorem~\ref{thm:weld-BM1}, the law of $(\widehat{\C},\phi,\eta\cup\ell,0,\infty)/\mathord\sim_\gamma$ is a constant multiple of
\begin{equation}
    \label{eq-proof-symmetry-RA-1}
    \begin{aligned}
        \int_{\R_+^4}\Weld(\QD_{1,1}(b),\QD_{0,4}(a,b,c,L),\cM^{\disk}_{0,2}(\frac{2}{3})(a,c),\QD_{1,1}(L))dadbdcdL.
    \end{aligned}
\end{equation}
Conversely, according to~\cite[Theorem B.5]{AHS21}, uniformly embedding the decorated quantum surface~\eqref{eq-proof-symmetry-RA-1} onto $\C$ gives $\LF_\C^{(\gamma,0),(\gamma,\infty)}\times\mathsf{n}$. 

Let $\wt\sn$ be the pushforward of $\sn$ under the inversion $z\mapsto\frac{1}{z}$. We claim that $\mathsf{n}=\wt\sn$. Indeed, let $\wt\eta$ and $\wt\ell$ be the image of $\eta$ and $\ell$ under $z\mapsto\frac{1}{z}$. The inversion invariance of $\SLE_{8/3}^{\sep}$ and the Brownian motion implies that $\wt\eta$ is also distributed as $\SLE_{8/3}^{\sep}$, and $\wt\ell$ is equal in law to the outer boundary seen from 0 of an independent Brownian motion starting from $\infty$ until first hitting $\wt\eta$. Hence, the law of the quantum surface $(\C,\phi,\wt\eta,\wt\ell,0,\infty)$ is given by~\eqref{eq-proof-symmetry-RA-1} as well, and uniformly embedding~\eqref{eq-proof-symmetry-RA-1} gives $\LF_\C^{(\gamma,0),(\gamma,\infty)}\times\wt\sn$. Hence, we find that $\sn=\wt\sn$.

For $\eta\cup\ell$ sampled from $\sn$, let $\eta'$ be the boundary of the connected component of $\wh\C\setminus(\eta\cup\ell)$ containing 0. Then the law of $(\eta,\eta')$ is invariant under $z\mapsto\frac{1}{z}$ due to the inversion invariance of $\sn$ above.
Let $D_1,D_2$ be the connected components of $\widehat{\C}\backslash(\eta\cup\tilde\eta)$ which contains 0 and $\infty$ respectively, and let $A$ be the annular connected component of $\widehat{\C}\backslash(\eta\cup\tilde\eta)$. By Propositions~\ref{prop-weld-loop} and~\ref{prop-weld-RA}, the joint law of $(D_1,\phi,0)/\mathord\sim_\gamma,(A,\phi)/\mathord\sim_\gamma,(D_2,\phi,\infty)/\mathord\sim_\gamma$ is a constant multiple of $\int_{\R_+^2}bL~\QD_{1,0}(b)\times\RA(b,L)\times\QD_{1,0}(L)dbdL$, which also equals $\int_{\R_+^2}bL~\QD_{1,0}(b)\times\RA(L,b)\times\QD_{1,0}(L)dbdL$ due to the inversion invariance of $(\eta,\eta')$. Therefore, we conclude that $\RA(b,L)=\RA(L,b)$.
\end{proof}

\begin{proof}[Proof of Proposition \ref{prop-LF-identify-RA}]
    Given Proposition \ref{prop-weld-RA} and Proposition \ref{prop-symmetry-RA}, the remaining part of proof is parallel to~\cite[Proposition 4.6]{ARS2022moduli} hence we will be brief.  Let $(\D,\phi,0,\eta)$ be an embedding of a sample from $\QD_{1,0}\otimes\sm$. Recall that under the reweighted measure $\frac{1}{|\QD_{1,1}(\nu_\phi(\eta))|}\QD_{1,0}\otimes\sm$, $(A_\eta,\phi)/\mathord\sim_\gamma$ has the law $\RA$. It suffices to show that condition on $\Mod(A_\eta)$ to be some $\tau>0$, the conditional law of $\phi|_{A_\eta}$ after uniformly embedding to $\mathcal{C}_\tau$ is $\LF_\tau$. According to~\cite[Proposition 2.13]{ARS2022moduli}, $\LF_\tau$ is characterized by its resampling property on two sides of boundaries. The resampling property on one side of $\phi$ follows from the resampling property of $\frac{1}{|\QD_{1,1}(\nu_\phi(\eta))|}\QD_{1,0}\otimes\sm$ due to the independence of the field and the curve, see~\cite[Lemma 4.7 and 4.8]{ARS2022moduli}. The resampling property on the other side follows from the symmetry of $\RA$ in Proposition~\ref{prop-symmetry-RA}. Then we conclude the proof.
\end{proof}

Once we know the joint law of the two boundary quantum lengths of $\RA$, we can obtain the explicit form of $m(d\tau)$ in Proposition~\ref{prop-LF-identify-RA} via the inverse Laplace transform. The following result is extracted in~\cite[Theorem 1.6]{ARS2022moduli}, which is based on~\cite[Theorem 1.3]{Wu22}.

\begin{proposition}
    \label{prop-LF-integrability}
For $\tau>0$ and a sample from $\LF_\tau$, let $b$ and $L$ be the quantum length of two boundaries. Then for $x\in\R$, we have
\begin{equation*}
    \LF_\tau[Le^{-L}b^{ix}]=\frac{\pi\gamma x\Gamma(1+ix)}{2\sinh(\frac{\gamma^2}{4}\pi x)}e^{-\frac{\pi}{4}\gamma^2\tau x^2}.
\end{equation*}
\end{proposition}

We now finish the proof of Proposition \ref{prop:LF-description} by combining Theorem~\ref{thm:weld-BM1} and Proposition~\ref{prop-LF-integrability}.
\begin{proof}[Proof of Proposition \ref{prop:LF-description}]
In the following, we fix $\gamma=\sqrt{8/3}$ and the constant $C$ can be varing from line to line.
We first compute $|\RA(b,L)|$ for $b,L>0$.  Note that by Definition \ref{def:RA} and Theorem \ref{thm:weld-BM1}, we have that $|\RA(b,L)|=C\int_{\R_+^2}|\QD_{0,4}(a,b,c,L)|\times|\cM^{\disk}_{0,2}(\frac{2}{3})(a,c)|dadc$ for some $C>0$. By \cite[Proposition 3.6]{AHS21}, we have $\int_0^{\ell}|\cM^{\disk}_{0,2}(\frac{2}{3})(a,\ell-a)|da=C\ell^{-\frac{1}{2}}$. Combined with Lemma \ref{lem-mass-QD}, we obtain that $|\RA(b,L)|=C(b+L)^{-2}$.\\

Now we can compute the explicit form of $m(d\tau)$ from Proposition~\ref{prop-LF-integrability}. Note that
\begin{equation}
\label{eq-proof-compute-m(dtau)-2}
    \RA[Le^{-L}b^{ix}]=C\int_{\R_+^2}\frac{Le^{-L}b^{ix}}{(b+L)^2}dbdL~
    \overset{b:=tL}{=\!=\!=\!=}~C\int_0^\infty L^{ix}e^{-L}dL\int_0^\infty\frac{t^{ix}}{(1+t)^2}dt
    =C\frac{\pi x\Gamma(1+ix)}{\sinh(\pi x)},~x\in\R.
\end{equation}
On the other hand, by Proposition~\ref{prop-LF-integrability} we have
\begin{equation}\label{eq-proof-compute-m(dtau)-2.5}
\RA[Le^{-L}b^{ix}]=\int_0^\infty\LF_\tau[Le^{-L}b^{ix}]m(d\tau)=\frac{\sqrt{6}\pi}{3}\frac{ x\Gamma(1+ix)}{\sinh(\frac{2}{3}\pi x)}\int_0^\infty e^{-\frac{2}{3}\pi\tau x^2}m(d\tau)
,~x\in\R
\end{equation}
Combining~\eqref{eq-proof-compute-m(dtau)-2} and~\eqref{eq-proof-compute-m(dtau)-2.5}, we find
\begin{equation}
\label{eq-proof-compute-m(dtau)-1}
\int_0^\infty e^{-\frac{2}{3}\pi\tau x^2}m(d\tau)=C\frac{\sinh(\frac{2}{3}\pi x)}{\sinh(\pi x)},~x\in\R.
\end{equation}
It remains to show that the Laplace transform of $f(\tau)d\tau$ in Proposition~\ref{prop:LF-description} also equals the right side of~\eqref{eq-proof-compute-m(dtau)-1}. Indeed, we have 
\begin{equation*}
    \int_0^\infty e^{-\frac{2}{3}\pi\tau x^2}f(\tau)d\tau=\frac{3}{2\pi}\sum_{n\ge1}(-1)^{n-1}\frac{n\sin(\frac{2\pi}{3}n)}{n^2+x^2}=\frac{3}{4}\frac{\sinh(\frac{2}{3}\pi x)}{\sinh(\pi x)}
\end{equation*}
where the last equality follows from \cite[Formula 1.445.4]{Table-of-integrals}. Hence we conclude the proof.
\end{proof}

\begin{remark}
\label{remark-qian}
For a simply connected domain $D$ containing $0$, denote $\CR(D;0)$ to be its conformal radius seen from 0. We remark that for $\eta$ sampled from $\sm$, one can further compute the joint law of $\CR(D_\eta;0)$ and $\tau=\Mod(A_\eta)$, following the approach in~\cite[Section 5.2]{ARS2022moduli}. Namely, for $\lambda_1,\lambda_2\in\R$ satisfy $4\lambda_2+1>2\lambda_1$ and $\lambda_2>-\frac{2}{3}$, we have
    \begin{equation*}
\sm\left[\left(e^{\pi\tau}\right)^{\lambda_1}\CR(D_\eta;0)^{\lambda_2}\right]=\frac{\sqrt{1-12\lambda_2}}{\sin\left(\frac{\pi}{3}\sqrt{1-12\lambda_2}\right)}\frac{\sin\left(\frac{\pi}{3}\sqrt{1-12\lambda_2+6\lambda_1}\right)}{\sin\left(\frac{\pi}{2}\sqrt{1-12\lambda_2+6\lambda_1}\right)}
    \end{equation*}
In particular, for $x\in(0,1)$, $\sm[\CR(D_\eta;0)\le x]=\sum_{n\ge1}(-1)^n\frac{16n^2}{\pi(4n^2-1)}x^{\frac{1}{3}n^2-\frac{1}{12}}$, which coincides with \cite[Theorem 1.5]{qian2019}.  
\end{remark}

\begin{remark}
    One can indeed evaluate all the constants in Theorems~\ref{thm:sm},~\ref{thm:weld-BM1} and Proposition~\ref{prop:LF-description}. However, since they will be absorbed into the expression of the non-disconnection probability $\mathbb P[G_\tau]$ (see~\eqref{eq:proba-c} below) and normalized such that $P(0)=1$, we do not need to specify them individually here.
\end{remark}

\section{Proof of Theorem~\ref{thm:main} via conformal restriction}\label{section: conformal restriction}

For general simply connected domain $D$ containing 0, let $\ell_D$ be the outer boundary of $B[0,\tau_D]$, and let $\sm_D$ be the law of the boundary of the connected component of $D\setminus\ell_D$ containing 0. Note that $\sm_D$ is conformally invariant, and $\sm_\D$ equals the measure $\sm$ defined above.

The main goal in this section is to establish the following relation between the non-disconnection probability $P(\tau):=\P[G_\tau]$ and the measure family $(\sm_D)$. Recall that $\Mod(\eta,D)$ is the conformal modulus of the annular domain between $\eta$ and $\partial D$.

\begin{proposition}\label{prop:restriction}
Let $D\subset D'\subset\C$ be two simply connected domain containing $0$. Then
\begin{equation*}
    \frac{d\sm_D}{d\sm_{D'}}(\eta)=\frac{P(\tau)/\tau}{P(\tau')/\tau'}{\bf 1}_{\eta\subset D},
\end{equation*}
here $\tau$ and $\tau'$ denotes for $\Mod(\eta,D)$ and $\Mod(\eta,D')$, respectively. 
\end{proposition}
Theorem~\ref{thm:main} then straightforwardly follows by combining Propositions~\ref{prop:restriction} and Theorem~\ref{prop: key lemma}.
\begin{proof}[Proof of Theorem~\ref{thm:main}, given Proposition~\ref{prop:restriction}]
Define $\sn_D(d\eta)=\frac{\tau}{P(\tau)}\sm_D(d\eta)$ for any simply connected domain $D$ containing 0. Note that $(\sn_D)$ is a family of conformally invariant measures on simple loops surrounding 0. Moreover, for any simply connected domain $D\subset D'$, ${\bf 1}_{\eta\subset D}\sn_{D'}(d\eta)=\sn_D(d\eta)$ according to Proposition~\ref{prop:restriction}. Hence, $(\sn_D)$ satisfies conformal restriction, and is equal to $(\SLE_{8/3, D}^\lp)$ by Proposition~\ref{prop:werner}. Combining with Theorem~\ref{prop: key lemma}, we obtain that
\begin{equation}\label{eq:proba-c}
P(\tau)=C\tau \frac{f(\tau)}{\eta(2i\tau)}
\end{equation}
for some constant $C>0$. The value of $C$ can be determined by the fact that the $\tau\to0$ limit on both sides of~\eqref{eq:proba-c} needs to be $1$. Note that $\eta(2i\tau)\sim\frac{1}{\sqrt{2\tau}}\exp(-\frac{\pi}{24\tau})$ as $\tau\to0$, as well as
    $$
    \begin{aligned}
        f(\tau)&=\frac{1}{2}\sum_{n\in\mathbb{Z}} (-1)^{n-1}n\sin\Big(\frac{2\pi}{3}n\Big)\exp\Big(-\frac{2\pi}{3}n^2\tau\Big)\\
        &=\frac{3\sqrt{6}}{8\tau^{\frac{3}{2}}}\sum\limits_{k\in\mathbb{Z}} \left(k+\frac{1}{6}\right)\exp\left(-\frac{3\pi}{2\tau}\left(k+\frac{1}{6}\right)^2\right)\sim \frac{\sqrt{6}}{16\tau^{3/2}}\exp(-\frac{\pi}{24\tau}),\quad \tau\to0.
    \end{aligned}
    $$
    Then we conclude that $C=\frac{8}{\sqrt{3}}$, as desired.
\end{proof}

The rest of this section is devoted to proving Proposition~\ref{prop:restriction}. Let $D$ be a simply connected domain containing $0$.
For a Brownian path $(B_t)_{0\le t\le \tau_D}$ starting from $0$ until hitting the boundary of $D$, we say $s\in(0,\tau_D]$ is a \emph{cut time} for $(B_t)_{0\le t\le \tau_D}$ if $B[0,s)\cap B(s,\tau_D]=\emptyset$. Note that by~\cite[Theorem 2.2]{Brownian-cut-point}, a.s.~for any $\varepsilon\in(0,\tau_D)$, there exists a cut time $s\in(0,\varepsilon)$ for $(B_t)_{0\le t\le \tau_D}$.

To relate $\sm(d\eta)$ with the non-disconnection probability $P(\tau)$, we need the following definition.

\begin{definition}\label{def:counting}
Let $t_0=\tau_D$. For each $n\ge1$, let $t_n$ be the supremum of $s\in(0,t_{n-1})$ such that $s$ is a cut point for $(B_t)_{0\le t\le \tau_D}$ and the outer boundary of $B[0,s]$ is a simple loop. For ${\bf{t}}$ sampled from the counting measure on $\{t_n\}_{n\ge1}$,  let $\wt\sm_D$ be the resulting law of the outer boundary of $B[0,{\bf{t}}]$.
\end{definition}
\begin{figure}[htbp]
    \centering
    \includegraphics[width=0.4\linewidth]{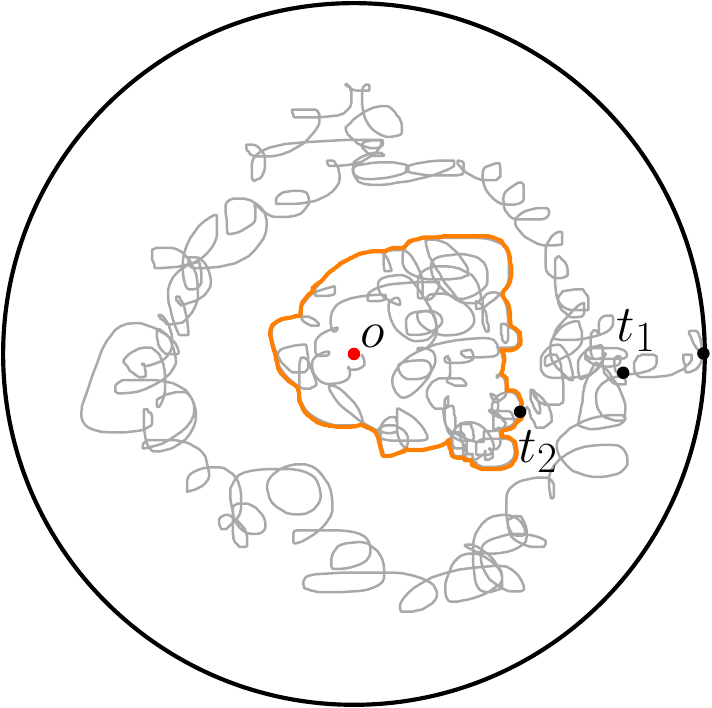}
    \caption{Illustration of Definition~\ref{def:counting}. The gray trajectory is $B[0,\tau_\D]$, and the orange loop is the outer boundary of $B[0,t_2]$.}
    \label{fig:counting}
\end{figure}

The proof of Proposition~\ref{prop:restriction} relies on the following two basic properties of $(\wt\sm_D)$.
\begin{lemma}\label{lem:derivative}
    We have $\frac{d\sm_D}{d\wt\sm_D}(\eta)=P(\tau)$ for $\tau=\Mod(\eta,D)$.
\end{lemma}

For $\eta$ sampled from $\wt\sm_D$, let $\nu_{\eta;D}(dz)$ be the conditional law of $B_{\tau_D}\in\partial D$ given $\eta$. The following lemma gives the restriction property for the family of measures $(\wt\sm_D)$.
\begin{lemma}
\label{lem: sm-restriction}
    Let $D,D'\subset\C$ be two simply connected domains containing $0$ such that $D\subset D'$. Then
\begin{equation*}
        \frac{d\big({\bf 1}_{\eta\subset D}\wt\sm_{D'}\big)}{d\wt\sm_D}(\eta)=\int_{\partial D} \P^z(\mathrm{hit}\ \partial D'\  \mathrm{ before\ hitting}\ \eta)\nu_{\eta;D}(dz)
\end{equation*}
Here $\P^z$ denotes the law of planar Brownian motion starting from $z$.
\end{lemma}

Intuitively, Lemma~\ref{lem:derivative} states that given a realization $\eta$ of $\wt\sm_D$, the ``conditional probability" of $\eta$ being the outermost layer equals the non-disconnecting probability on the annulus $A(\eta,\partial D)$. Lemma~\ref{lem: sm-restriction} is of the similar fashion, saying that given a realization $\eta$ of $\wt\sm_D$, then $\eta$ is still in the support of $\wt\sm_{D'}$ if and only if the Brownian motion after first hitting $\partial D$ does not intersect $\eta$ before hitting $\partial D'$. However, though the above explanation seems straightforward, it involves some technical issues to make precise sense of conditioning on a zero-probability event. The proofs are based on techniques from~\cite{Brownian-beads}, and we postpone them to Section~\ref{sec:pf-derivative}.

\subsection{Proof of Proposition~\ref{prop:restriction}}

In this section we finish the proof of Proposition~\ref{prop:restriction}, based on Lemmas~\ref{lem:derivative} and~\ref{lem: sm-restriction} above.

\begin{proof}[Proof of Proposition~\ref{prop:restriction}]
By Theorem~\ref{prop: key lemma} and Lemma~\ref{lem:derivative}, we see that
\begin{equation}
\label{eq-derivative-wt-sm-SLE}
    \frac{d\wt\sm_D}{d\SLE_{8/3,D}^{\lp}}(\eta)=\frac{f(\tau)}{P(\tau)\eta(2i\tau)}=:H(\tau),\quad \tau=\Mod(\eta,D).
\end{equation}
In particular, we have
\begin{equation}\label{eq:restriction1}
\frac{d\wt\sm_D}{d(1_{\eta\subset D}\wt\sm_{D'})}(\eta)=\frac{H(\tau)}{H(\tau')}\quad \mathrm{for}\ \tau'=\Mod(\eta,D').
\end{equation}
In the following, we show that $H(\tau)=\frac{C}{\tau}$ for some constant $C\in(0,\infty)$, and Proposition~\ref{prop:restriction} then follows by combining~\eqref{eq:restriction1} and Lemma~\ref{lem:derivative}.

Let $\D_r:=e^{2\pi r}\D$ for $r\in\R$. For $\rho>0$, by Lemma~\ref{lem: sm-restriction}, we have
\begin{equation}\label{eq:restriction2}
    \frac{d\big(1_{\eta\subset\D}\wt\sm_{\D_\rho}\big)}{d\wt\sm}(\eta)=\int_{\partial\D} \P^z(\text{hit }\partial\D_\rho \text{ before }\eta)\nu_{\eta;\D}(dz).
\end{equation}
Comparing~\eqref{eq:restriction2} with ~\eqref{eq:restriction1}, we see for any $\eta\subset\D$,
\begin{equation}\label{eq:H-function}
    H(\Mod(\eta;\D_\rho))=\int_{\partial\D} \P^z(\text{hit }\partial\D_\rho \text{ before }\eta)\nu_{\eta;\D}(dz)H(\Mod(\eta;\D)).
\end{equation}
In particular, for $\tau>0$ and $\delta\in(0,\tau)$, if we choose $\eta$ to be contained in $\D_{-\tau+\delta}\backslash\D_{-\tau}$, then
\begin{equation*}
        \tau-\delta\le \Mod(\eta;\D)\le \tau,\quad \tau-\delta+ \rho\le\Mod(\eta;\D_\rho)\le\tau+\rho,
\end{equation*}
and hence the probability
$$\frac{\tau-\delta}{\tau-\delta+\rho}\le\P^z(\text{hit }\partial\D_\rho\text{ before } \eta)\le\frac{\tau}{\tau+\rho}.$$
Plugging above into~\eqref{eq:H-function} and using that $\int_{\partial\D}\nu_{\eta,\D}(dz)=1$, we find
\[
\inf_{x\in[\tau-\delta+\rho,\tau+\rho]}H(x)\le \frac{\tau}{\tau+\rho}\sup_{x\in[\tau-\delta,\tau]}H(x)\quad {\rm and}\quad \sup_{x\in[\tau-\delta+\rho,\tau+\rho]}H(x)\ge \frac{\tau-\delta}{\tau-\delta+\rho}\inf_{x\in[\tau-\delta,\tau]}H(x).
\]
Taking $\delta\to0$, the continuity of $H$ then implies $H(\tau+\rho)=H(\tau)\frac{\tau}{\tau+\rho}$. Since $\tau$ and $\rho$ are arbitrary, we conclude that $\tau H(\tau)=C$ for some constant $C\in(0,\infty)$.
\end{proof}

As a consequence, we can further give the explicit description of the conditional law $\nu_{\eta;\D}(dz)$.
\begin{corollary}
\label{cor-condition-law-point}
     For $\eta\subset\D$ surrounding 0, let $f:A_\eta\to\A_\tau$ be a conformal map preserving $\S^1$. Then $\nu_{\eta;\D}(dz)=\frac{1}{2\pi}|f'(z)|dz$.
\end{corollary}
\begin{proof}
Similar to~\eqref{eq:H-function}, for any simply connected domain $D\supset\D$, we have that for $\eta\subset\D$,
    $$
    H(\Mod(\eta;D))=\int_{\partial\D} \P^z(\text{hit }\partial D \text{ before }\eta)\nu_{\eta;\D}(dz)H(\Mod(\eta;\D)).
    $$
Combining with that $H(\tau)=\frac{C}{\tau}$ for some $C\in(0,\infty)$, we find
\begin{equation}
\label{eq-identity-hitting-prob-1}
    \int_{\partial\D} \P^z(\text{hit }\partial D \text{ before }\eta)\nu_{\eta;\D}(dz)=\frac{\Mod(\eta;\D)}{\Mod(\eta;D)}.
\end{equation}
On the other hand, note that
\begin{equation}
\label{eq-identity-hitting-prob-2}
    \frac{1}{2\pi}\int_{\partial\D} \P^z(\text{hit }\partial D \text{ before }\eta)|f'(z)|dz=\frac{\Mod(\eta;\D)}{\Mod(\eta;D)}.
\end{equation}
We can justify~\eqref{eq-identity-hitting-prob-2} as follows. 
Recall the definition of Brownian excursion measures in Section~\ref{section 2.1}.
Let $W$ be an excursion path on $A_{\eta,\D}$ sampled from $\BE(A_{\eta,\D})$, and denote its endpoint on $\S^1$ by $\textbf{z}$. If we further sample an independent Brownian motion $W'$ from $\textbf{z}$ and restrict on the event that $W'$ hits $\partial D$ before $\eta$, then by the strong Markov property, the concatenation of $W$ and $W'$ is distributed as the Brownian excursion $\BE(A_{\eta,D})$ on $A_{\eta,D}$. Moreover, the conformal covariance of Poisson kernel yields that the law of $\textbf{z}$ under $\BE(A_{\eta,\D})$ is $\frac{1}{2\pi\Mod(\eta;\D)}|f'(z)|dz$ (recall that $\frac{1}{\Mod(\eta;\D)}$ is the total mass of $\BE(A_{\eta,\D})$). Combined, we obtain~\eqref{eq-identity-hitting-prob-2}.

Now the corollary follows by comparing \eqref{eq-identity-hitting-prob-1} and \eqref{eq-identity-hitting-prob-2} due to the arbitrariness of $D$.
\end{proof}

\subsection{Proof of Lemmas~\ref{lem:derivative} and~\ref{lem: sm-restriction}}\label{sec:pf-derivative}

Throughout this section, let $\gamma$ be any Jordan loop in $D$ surrounding 0, and $A_\varepsilon:=\{z\in\C:{\rm dist}(z,\gamma)<\varepsilon\}$ be the $\varepsilon$-neighborhood of $\gamma$. Denote $\mathcal{A}_\varepsilon$ to be the collection of simple loops surrounding $0$ and contained in $A_\varepsilon$. Let $\gamma_\varepsilon^o$ and $\gamma_\varepsilon^i$ be the outer and inner boundaries of $A_\varepsilon$, respectively. For a planar Brownian motion $(B_t)_{t\ge0}$, we write $\mathcal{F}_t:=\sigma((B_s)_{0\le s\le t})$ to be its corresponding filtration.

We first prove Lemma~\ref{lem:derivative}. For a Brownian motion $(B_t)_{0\le t\le\tau_D}$ from 0 until hitting $\partial D$, let $\{\ell_n\}$ be the collection of outer boundaries of $B[0,t_n]$ in Definition~\ref{def:counting}. Let $E_\varepsilon:=\{\ell_1\subset A_\varepsilon\}$, and $\wt E_\varepsilon$ be the event that there exists $n\ge1$ such that $\ell_n\subset A_\varepsilon$. Note that
\begin{equation*}
\wt\sm(\mathcal{A}_\varepsilon)=\P[\wt E_\varepsilon](1+o_\varepsilon(1)),\quad \sm(\mathcal{\mathcal{A}_\varepsilon})=\P[E_\varepsilon].
\end{equation*}
So it suffices to consider the conditional probability $\P[E_\varepsilon|\wt E_\varepsilon]$.
\begin{lemma}\label{lem:estimate}
We have $\P[E_\varepsilon|\wt E_\varepsilon]=P(\tau)(1+o_\varepsilon(1))$ as $\varepsilon\to0$ for $\tau=\Mod(\gamma,D)$.
\end{lemma}
\begin{proof}
Let $(\mathcal{G}_t)$ be the cut time filtration defined in~\cite{Brownian-beads}, i.e. $\mathcal{G}_t$ is generated by $\mathcal{F}_t$ and the set $G_t$ of all cut times of $(B_t)_{0\le t\le\tau_D}$ before $t$.
Consider the first cut time ${\rho}$ of $(B_t)_{0\le t\le \tau_D}$ after hitting $\gamma_\varepsilon^i$ such that the outer boundary of $B[0,{\rho}]$ is a simple loop (if it not exists, define ${\rho}$ to be $\tau_D$). Then $\rho$ is a stopping time of $(\mathcal{G}_t)$, and $\wt E_\varepsilon$ is measurable w.r.t. $\mathcal{G}_{\rho}$. Denote $A_{\rho}$ to be the connected component of $D\setminus B[0,{\rho}]$ with $\partial D$ being part of its boundary. Following the same lines as in~\cite[Proposition 14]{Brownian-beads}, the conditional law of $B[{\rho},\tau_D]$ given $\mathcal{G}_{\rho}$ is a Brownian excursion on $A_{\rho}$ starting from $B_{\rho}$ and conditioned to hit $\partial D$, which is denoted by $(W_s)$ in the following. (\cite{Brownian-beads} proved the result for Brownian excursions between boundary points, but the same proof works in our case as well.)

Now consider the conditional probability $\P[E_\varepsilon|\mathcal{G}_{\rho}](\omega)$ for $\omega\in \wt E_\varepsilon$. According to the conditional law of $B[{\rho},\tau_D]$ described above, if we define
\begin{equation*}
\begin{aligned}
&G_1:=\{(W_s) \text{ does not disconnect the outer boundary of } B[0,{\rho}] \text{ from }\partial D\},
\\&G_2:=\{(W_s) \text{ does not disconnect } \gamma_\varepsilon^o \text{ from } \partial D \text{ after the last hitting time of }\gamma_\varepsilon^o\},
\end{aligned}
\end{equation*}
then we have
\begin{equation*}
P_\omega(G_1)\le \P[E_\varepsilon|\mathcal{G}_{\rho}](\omega)\le P_\omega(G_2)\quad \text{for any } \omega\in \wt E_\varepsilon,
\end{equation*}
here $P_\omega$ denotes the probability measure of $(W_s)$ (which depends on $\omega$). On the other hand, let $\tau_\varepsilon^o:=\Mod(\gamma_\varepsilon^o,\partial D)$ and $\tau_\varepsilon^i:=\Mod(\gamma_\varepsilon^i,\partial D)$. Note that $ P_\omega(G_1)\ge P(\tau_\varepsilon^i)$ by the monotonicity of $\tau\mapsto P(\tau)$, while $P_\omega(G_2)=P(\tau_\varepsilon^o)$ since $(W_s)$ after last hitting $\gamma_\varepsilon^o$ is a Brownian excursion on the annular region between $\gamma_\varepsilon^o$ and $\partial D$. Therefore, we obtain $P(\tau_\varepsilon^i)\le \P[E_\varepsilon|\wt E_\varepsilon]\le P(\tau_\varepsilon^o)$. Finally, since both $P(\tau_\varepsilon^i)$ and $P(\tau_\varepsilon^o)$ equals $P(\tau)(1+o_\varepsilon(1))$, the result follows.\qedhere

\end{proof}
\begin{proof}[Proof of Lemma~\ref{lem:derivative}]
By Lemma~\ref{lem:estimate}, we find $\frac{\sm(\mathcal{A}_\varepsilon)}{\wt\sm(\mathcal{A}_\varepsilon)}=P(\tau)(1+o_\varepsilon(1))$ for $\tau=\Mod(\gamma,D)$. Since $\gamma$ is arbitrary, we conclude.
\end{proof}

Now we prove Lemma~\ref{lem: sm-restriction} in a similar fashion.
\begin{proof}[Proof of Lemma~\ref{lem: sm-restriction}]
Consider the Brownian motion $(B_t)_{0\le t\le\tau_{D'}}$ start from $0$ until hitting $\partial D'$. Let $\tau_D$ be the hitting time of $\partial D$.
Let $\{t_n\}_{n\ge 1}$ and $\{t_n'\}_{n\ge 1}$ be the decreasing sequences for $(B_t)_{0\le t\le\tau_D}$ and $(B_t)_{0\le t\le\tau_D'}$ as in Definition~\ref{def:counting}. Let $\ell_n,\ell_n'$ be the outer boundaries of $B[0,t_n]$ and $B[0,t_n']$.
Let $\bm{\eta}$ be sampled from the counting measure on $\{\ell_n\}_{n\ge 1}$ (hence the marginal law of $\bm\eta$ is $\wt\sm_D$).

Recall that $A_\varepsilon\subset D$ is the $\varepsilon$-neighborhood of a Jordan loop $\gamma$ surrounding $0$. Let $\wt E_\varepsilon$ (resp. $\wt E_\varepsilon'$) denote the event that there exists a loop in $\{\ell_n\}$ (resp.$\{\ell_n'\}$) contained in $A_\varepsilon$. The strong Markov property then gives that for $\omega\in \wt E_\varepsilon$ and $\textbf{z}:=B_{\tau_D}$,
\begin{align}\label{eq:smp1}
\P^\textbf{z}(\mathrm{hit}\ \partial D'\  \mathrm{ before\ hitting}\ \gamma_\varepsilon^o)\le \P[\wt E_\varepsilon'|\mathcal{F}_{\tau_D}](\omega)\le \P^\textbf{z}(\mathrm{hit}\ \partial D'\  \mathrm{ before\ hitting}\ \gamma_\varepsilon^i).
\end{align}
Note that ${\bm\eta}\in \mathcal{A}_\varepsilon$ implies that $\wt E_\varepsilon$ holds. Hence for $z\in\partial D$,~\eqref{eq:smp1} implies that
\begin{equation}\label{eq:smp2}
\begin{aligned}
\P^z(\mathrm{hit}\ \partial D'\  \mathrm{ before\ hitting}\ \gamma_\varepsilon^o)-o_\delta(1)&\le \P\otimes{\rm Count}[\wt E_\varepsilon'|{\bm\eta}\in \mathcal{A}_{\varepsilon}, |\textbf{z}-z|<\delta]\\
&\le \P^z(\mathrm{hit}\ \partial D'\  \mathrm{ before\ hitting}\ \gamma_\varepsilon^i)+o_\delta(1),
\end{aligned}
\end{equation}
here we use the continuity of $w\mapsto\P^w(\mathrm{hit}\ \partial D'\  \mathrm{ before\ hitting}\ \gamma_\varepsilon^q), q\in\{i,o\}$.
By the definition of $\nu_{\eta,D}(dz)$ and using this continuity again, let $\delta\to0$, then~\eqref{eq:smp2} yields
\begin{equation}\label{eq:smp3}
\begin{aligned}
\int_{\mathcal{A}_\varepsilon}\int_{\partial D}\P^z(\mathrm{hit}\ \partial D'\  \mathrm{ before\ hitting}\ \gamma_\varepsilon^o)&\nu_{\eta,D}(dz)\wt\sm_D(d\eta)\le \P\otimes{\rm Count}[\wt E_\varepsilon',{\bm\eta}\in \mathcal{A}_{\varepsilon}]\\&\le \int_{\mathcal{A}_\varepsilon}\int_{\partial D}\P^z(\mathrm{hit}\ \partial D'\  \mathrm{ before\ hitting}\ \gamma_\varepsilon^i)\nu_{\eta,D}(dz)\wt\sm_D(d\eta).
\end{aligned}
\end{equation}
Since $\P\otimes{\rm Count}[\wt E_\varepsilon',{\bm\eta}\in \mathcal{A}_{\varepsilon}]=\wt\sm_{D'}(\mathcal{A}_\varepsilon)(1+o_\varepsilon(1))$, we obtain from~\eqref{eq:smp3} that
$$\frac{\wt\sm_{D'}(\mathcal{A}_\varepsilon)}{\wt\sm_{D}(\mathcal{A}_\varepsilon)}=(1+o_\varepsilon(1))\int_{\partial D} \P^z(\mathrm{hit}\ \partial D'\  \mathrm{ before\ hitting}\ \gamma)\nu_{\gamma;D}(dz).$$
As in the proof of Lemma~\ref{lem:derivative}, by varying $\gamma$ and $\varepsilon$, the result then follows.
\end{proof}

\section{Matter-coupled Brownian annulus}
\label{section-weld-matter-BA}

As explained in Section~\ref{subsection-proof-strategy}, the intuition behind our proof of Theorem~\ref{thm:main} comes from the Brownian excursion decorated Brownian annulus conditioned on  not to disconnect the two boundaries, which we called the  matter-coupled Brownian annulus. 
In this section we show that  the matter-coupled Brownian annulus can indeed be obtained by removing the small disk in the Brownian motion decorated Brownian disk in Figure~\ref{fig:weld-BM1}. We will give a precise formulation of this statement and its proof in Section~\ref{subsec:BA}. 
Our proof is based on Theorem~\ref{prop-bdy-BM}, which we prove in Section~\ref{subsec:SLEBM}.

\subsection{Brownian motion and SLE$_{8/3}$ loop on the disk: proof of Theorem~\ref{prop-bdy-BM}}\label{subsec:SLEBM}

The proof of Theorem~\ref{prop-bdy-BM} is based on the following two lemmas, both of which rely on the analysis in Section~\ref{sec:pf-derivative}. Suppose $\bf{t}$ is sampled from the counting measure on $\{t_n\}_{n\ge1}$.
For $t\in(0,\tau_{\D})$, let $A_t$ be the annular connected component of $\D\backslash B[0,t]$.

\begin{lemma}
    \label{lem-condtion-wt-sm}
    Conditioned on $\partial^oB[0,\bf{t}]$ and $B_{\bf{t}}$, the conditional law of $(B_t)_{{\bf{t}}\le t\le\tau_{\D}}$ is the Brownian excursion on $A_{\bf{t}}$ starting from $B_{\bf{t}}$ and conditioned to hit $\S^1$.
\end{lemma}
\begin{proof}
    Denote the law of $((B_t)_{0\le t\le\tau_\D},{\bf t})$ to be $\P\times{\rm Count}$.
    Let $f_{\bf{t}}$ be the conformal map from $A_{\bf{t}}$ to $\A_{\tau}$ such that 
    $f_{\bf{t}}(1)=e^{-2\pi\tau}$.
    Let $F$ be any bounded measurable function of $B[\bf{t},\tau_{\D}]$.
    Then for the $\varepsilon$-neighborhood $A_\varepsilon$ of any simple loop $\gamma\subset\D$ surrounding $0$ and any open subset $I$ of $\S^1$, we have
    \begin{equation*}
        \begin{aligned}
            \P\times{\rm Count}[F(B[{\bf{t}},\tau_{\D}])\mathbf{1}_{\partial^oB[0,{\bf{t}}]\subset A_\varepsilon,f_{\bf{t}}(B_{\bf{t}})\subset I}]
            &=\sum\limits_{n\ge1}\E[F(B[t_n,\tau_{\D}])\mathbf{1}_{\partial^oB[0,t_n]\subset A_{\varepsilon},f_{t_n}(B_{t_n})\subset I}]\\
            &=\E[F(B[\rho,\tau_{\D}])\mathbf{1}_{\partial^oB[0,\rho]\subset A_{\varepsilon},f_{\rho}(B_{\rho})\subset I}](1+o_{\varepsilon}(1)),
        \end{aligned}
    \end{equation*}
    where $\rho$ is defined in the proof of Lemma~\ref{lem:estimate}.
 Here, the last line follows from $\P[\exists! n\ge1 \ {\rm s.t.}\ \partial^oB[0,t_n]\subset A_{\varepsilon}|\wt{E}_{\varepsilon}]=1-o_\varepsilon(1)$ as $\varepsilon\to0$.
    Since $\partial^oB[0,\rho]$ and $B_{\rho}$ is measurable with respect to the cut time filtration $\mathcal{G}_{\rho}$, the conditional law of $(B_t)_{\rho\le t\le \tau_{\D}}$ given $\partial^oB[0,\rho]$ and $B_{\rho}$ is the Brownian excursion in $A_{\rho}$ from $B_{\rho}$ to $\S^1$.  Hence we conclude the proof by the arbitrariness of $\gamma$ and $I$.
\end{proof}

\begin{lemma}
    \label{lem-condition-law-point}
      Let $g_{\bf{t}}:A_{\bf{t}}\to\A_{\tau}$ be the conformal map with $g_{\bf{t}}(1)=1$. Then given $\partial^oB[0,\bf{t}]$, the conditional law of $g_{\bf{t}}(B_{\bf{t}})$ is the uniform probability measure on $e^{-2\pi\tau}\S^1$.
\end{lemma}
\begin{proof}
    Let $\mu(d\phi)$ be the conditional law of $\Phi=\arg g_{\bf{t}}(B_{\bf{t}})$ given $\partial^oB[0,\bf{t}]$. Then according to Lemma~\ref{lem-condtion-wt-sm}, the conditional density of $\Theta=\arg g_{\bf{t}}(B_{\tau_{\D}})$ given $\partial^oB[0,\bf{t}]$ is proportional to $\int_{0}^{2\pi}K(\theta-\phi)\mu(d\phi)$, where $K(\theta-\phi):=H_{\A_{\tau}}(e^{-2\pi\tau}e^{i\phi},e^{i\theta})$. On the other hand, by Corollary~\ref{cor-condition-law-point}, the conditional law of $\Theta$ given $\partial^oB[0,\bf{t}]$ is uniform on $[0,2\pi)$. Therefore,
    for any $\theta\in[0,2\pi)$, we have $\int_{0}^{2\pi}K(\theta-\phi)\mu(d\phi)=C_\tau$ where $C_\tau>0$ only depends on $\tau$.  In particular, its Fourier transform with respect to $\theta$ vanishes, i.e.
    \begin{equation}\label{eq:merge}
    \int_0^{2\pi}\int_0^{2\pi}K(\theta-\phi)\mu(d\phi)e^{-in\theta}d\theta=0,\quad \forall n\in\Z_+. 
    \end{equation}
    
    It remains to show that $\mu(d\phi)$ is uniform on $[0,2\pi)$.
    Note that for $n\in\Z_+$, we have $\int_0^{2\pi}K(\alpha)e^{-in\alpha}d\alpha=C\frac{e^{2\pi\tau}n}{\sinh(2\pi\tau n)}$ for some constant $C>0$ (see e.g.~\cite[Section 3.2]{Law11}). Hence,
    $$
    \int_0^{2\pi}\int_0^{2\pi}K(\theta-\phi)\mu(d\phi)e^{-in\theta}d\theta=C\frac{e^{2\pi\tau}n}{\sinh(2\pi\tau n)}\int_0^{2\pi}e^{-in\phi}\mu(d\phi).
    $$
    Compared with~\eqref{eq:merge}, we find $\int_0^{2\pi}e^{-in\phi}\mu(d\phi)=0$ for all $n\in\Z_+$, which concludes the proof.
\end{proof}

We now finish the proof of Theorem~\ref{prop-bdy-BM}. Note that the measure $\nu_\eta$ in the statement of Theorem~\ref{prop-bdy-BM} equals the restriction of $\BE(A_\eta)$ to the paths that do not disconnect $\eta$ from $\S^1$.
 \begin{proof}[Proof of Theorem~\ref{prop-bdy-BM}]
    By Lemma~\ref{lem-condtion-wt-sm} and~\ref{lem-condition-law-point}, the conditional law of $(B_t)_{{\bf{t}}\le t\le\tau_{\D}}$ given $\partial^oB[0,\bf{t}]$ is $\BE(A_{\bf{t}})^{\#}$. Moreover, by~\eqref{eq-derivative-wt-sm-SLE} and that $H(\tau)=\frac{C}{\tau}$, the marginal law of $\partial^oB[0,{\bf{t}}]$ (i.e.~$\wt\sm$) is a constant multiple of $|\BE(A_{\eta})|\SLE^{\lp}_{8/3,\D}(d\eta)$ (recall that $|\BE(A)|=\frac{1}{\Mod(A)}$). By restricting the law of $(\partial^oB[0,{\bf{t}}],(B_t)_{{\bf{t}}\le t\le\tau_{\D}})$ on the event that $B[\bf{t},\tau_{\D}]$ does not disconnect $\partial^oB[0,{\bf{t}}]$ from $\S^1$ (i.e.~${\bf{t}}=t_1$), the result then follows.
\end{proof}

\subsection{Matter-coupled Brownian annulus inside the Brownian disk} \label{subsec:BA}
We first give the precise definition of matter-coupled Brownian annulus. Consider a measure $\mathbb M(d\phi, d\tau)$ such that the law of $(\A_\tau,\phi)/\mathord\sim_\gamma$ is  $\BA$
from Definition~\ref{def:BA}.
Given $\tau>0$, let $\BE_\tau=\BE(\A_\tau)$ be the law of a Brownian excursion on $\A_\tau$. Let $\mathcal G$ be the event that a sample $W$ from $\BE_\tau$ does not disconnect $e^{-2\pi\tau}\S^1$ from $\S^1$. Sample $(\tau, \phi, W)$ from  ${\bf 1}_{\mathcal G} \BE_\tau (d W)   \mathbb M(d\phi, d\tau)$.  Let $\tBA$  be the law of $(\A_\tau,\phi,W)/\mathord\sim_\gamma$, which we call the \emph{matter-coupled Brownian annulus}. 
Note that this definition does not rely on the choice of $\mathbb M(d\phi, d\tau)$  due to the conformal invariance of $W$.

Next we define another Brownian motion decorated quantum surface with annular topology from Brownian motion decorated Brownian disk. Let $(B_t)_{t\ge0}$ be a Brownian motion starting from 0 and $\tau_{\D}$ is its hitting time on $\S^1$ and denote the law of $(B_t)_{0\le t\le\tau_\D}$ by $\P$.
Recall the cut times $(t_n)_{n\ge1}$ for 
$(B_t)_{t\in[0,\tau_{\D}]}$ in Definition~\ref{def:counting}. Let $\eta$ be the outer boundary of $B[0,t_1]$ and $A_\eta$ be the annular connected component of $\D\backslash\eta$. 
Suppose $(\D,\phi,0)$ is an embedding of an independent sample from $\QD_{1,0}$ and $\nu_\phi(\eta)$ is the quantum length of $\eta$. Then we define $\wt{\RA}$ to be the law of $(A_\eta,\phi,B[t_1,\tau_{\D}])/\mathord\sim$ under the reweighted measure $\frac{1}{|\QD_{1,1}(\nu_\phi(\eta))|}(\QD_{1,0}\times\P)$. Note that the measure $\RA$ in Definition~\ref{def:RA} can be obtained from a sample from $\wt{\RA}$ by forgetting the decorated Brownian path.

\begin{theorem}\label{thm-weld-matter-BA}\label{thm-equivalence-two-annuli}
    There exists a constant $C>0$, such that $\tBA=C\wt{\RA}$.
\end{theorem}

\begin{proof}
Let $(\D,\phi,0)$ be an embedding of an independent sample from $\QD_{1,0}$, and $(B_t)_{0\le t\le\tau_\D}$ be an independent Brownian motion on $\D$. Denote $\eta=\partial^o B[0,t_1]$, and $W$ to be the Brownian path $B[t_1,\tau_\D]$. Note that the law of $(A_{\eta},\phi,W)/\mathord\sim_\gamma$ under $\frac{1}{|\QD_{1,1}(\nu_\phi(\eta))|}\QD_{1,0}\times\P$ is $\wt{\RA}$.
On the other hand, by Theorem~\ref{prop-bdy-BM}, the law of $(\eta,W)$ equals a constant multiple of ${\bf 1}_\mathcal{G}\BE(A_\eta)\SLE_{8/3,\D}(d\eta)$ (here $\mathcal{G}$ is the event that $W$ does not disconnect $\eta$ from $\S^1$). Combined with Proposition~\ref{thm:weld-BA}, we find that the law of $(A_{\eta},\phi,W)/\mathord\sim_\gamma$ under $\frac{1}{|\QD_{1,1}(\nu_\phi(\eta))|}\QD_{1,0}\times\P$~is also a constant multiple of $\tBA$. Hence, we conclude that $\tBA=C\wt{\RA}$ for some constant $C>0$.
\end{proof}

As a corollary, we get the following conformal welding results on $\tBA$.

\begin{corollary}
    \label{thm-weld-matter-BA}
Let $(\A_\tau,\phi,W)$ be an embedding of a sample from $\tBA$ and $\eta'$ be the outer boundary of $W$. Then there exists a constant $C>0$, such that $(\A_\tau,\phi,\eta')/\mathord\sim$ has the same law as
    \begin{equation}
        \label{eq-weld-matter-BA}
        C\int_{\R_+^4}\Weld\left(\QD_{0,4}(a,b,c,\ell),\cM^{\disk}_{0,2}\left(\frac{2}{3}\right)(a,c)\right)dadbdcd\ell.
    \end{equation}
\end{corollary}

\begin{proof}
    By Theorem~\ref{thm:weld-BM1} and the definition of $\wt{\RA}$, if $(\A_\tau,\phi,W)$ is an embedding of a sample from $\wt{\RA}$ and $\eta'$ is the outer boundary of $W$, then there exists a constant $C>0$ such that $(\A_\tau,\phi,\eta')/\mathord\sim$ has the same law as in~\eqref{eq-weld-matter-BA}. We then conclude the proof by Theorem~\ref{thm-equivalence-two-annuli}.
\end{proof}

\begin{remark}
It is clear in the discrete that inside a random walk decorated random triangulation of a disk, one can find a sample of the random walk excursion decorated random triangulation of an annulus, after a proper reweighting of the boundary distribution of the annulus. It is less clear that the scaling limit of  the complement of this annulus is a Brownian disk.
\end{remark}

\appendix 

\section{Random walk formulation of the non-disconnection probability}\label{appendix_A}
We provide a proof of Equation~\eqref{eq-scaling} that expresses the non-disconnection probability via the scaling limit of simple random walk.
        
\begin{proof}[Proof of Equation~\eqref{eq-scaling}]
    Let $\mathcal{K}$ be the space of continuous curves $\gamma:[0,t_\gamma]\to\C$ connecting the inner boundary $e^{-2\pi\tau}\S^1$ and the outer boundary $\S^1$ of $\A_\tau$, with the metric $d$ defined by 
    \begin{equation}\label{eq:metric}
        d(\gamma,\gamma'):=\sup_{0\le s\le 1}|\gamma(t_\gamma s)-\gamma'(t_{\gamma'}s)|+|t_\gamma-t_{\gamma'}|,
    \end{equation}
     for any $\gamma, \gamma'\in \mathcal{K}$. 
     Let $\mathcal{G}$ be the collection of curves $\gamma\in\mathcal{K}$ that does not disconnect the two boundaries of $\A_\tau$ (which is an open set in $(\mathcal{K},d)$).
     Then we have $\P[G_\tau]=\BE(\A_\tau)^\#[\mathcal{G}]$. In the following, we prove~\eqref{eq-scaling} by establishing the convergence of the corresponding random walk excursion measures. We view a random walk path as a continuous curve by linear interpolation.
     
     Let $\mu_\delta^\#$ be the probability measure on random walk excursions connecting the two boundaries on $\A_{\tau,\delta}$; namely, for $\omega^\delta\in\Omega^\delta$, we define $\mu_\delta^\#[\omega^\delta]=\frac{1}{Z_\delta}(\frac{1}{4})^{|\omega^\delta|}$. We first show that $\mu_\delta^\#$ weakly converges to $\BE(\A_\tau)^\#$ as $\delta\to0$. Let $\mathcal{K}'$ be the space of continuous curves in $\D$ from 0 until hitting $\S^1$, with the same metric as~\eqref{eq:metric}. Note that the simple random walk $(S_{\delta^{-2}t})$ on $\delta\Z^2$ starting from 0 until hitting $\S^1$ converges in law to $(B_t)_{0\le t\le\tau_\D}$, and the laws of their trajectories after last hitting $e^{-2\pi\tau}\S^1$ are given by $\mu_\delta^\#$ and $\BE(\A_\tau)^\#$, respectively. For $\gamma\in \mathcal{K}$, let $T_\tau=T_\tau(\gamma):=\sup\{t>0:|\gamma(t)|=e^{-2\pi\tau}\}$, and define the function $F:\mathcal{K}'\to\mathcal{K}$ such that $F(\gamma)=\gamma[T_\tau,t_\gamma]$. Then for a.s.~Brownian excursion path $W$ sampled from $\BE(\A_\tau)^\#$, since $W[T_\tau-\varepsilon,T_\tau]$ intersects with $e^{-2\pi\tau}\D$ for any $\varepsilon>0$\footnote{Note that $T_\tau$ is a stopping time for the time reverse process of $W$. By the strong Markov property, given $W[T_\tau,t_W]$, $(W_{T_\tau-s})_{s\in[0,T_\tau]}$ is then distributed as a Brownian bridge from $W(T_\tau)$ to $0$  on $\D$, which immediately hits $e^{-2\pi\tau}\D$.}, we find that $W$ is in the continuity set of $F$. Therefore, by considering the pushforward measure under $F$ of the simple random walk $(S_{\delta^{-2}t})$ above, we obtain the desired convergence.

    Note that the right side of~\eqref{eq-scaling} is equal to $\mu_\delta^\#[\mathcal{G}]$. Hence, it suffices to justify that $\mathcal{G}$ is a continuity set for $\BE(\A_\tau)^\#$.
    Suppose $W$ is sampled from $\BE(\A_\tau)^\#$ conditioned on $\mathcal{G}^c$ (i.e.~$W$ is conditioned to disconnect the two boundaries of $\A_\tau$), and define its first disconnecting time $\sigma:=\inf\{t>0:W[0,t] ~\textrm{disconnects}~e^{-2\pi\tau}\S^1~\textrm{from} ~\S^1\}$. We also define $\wt\sigma:=\sup\{0<t<\sigma:W(t)=W(\sigma)\}$. Note that $\sigma$ is a stopping time for $W$. Then by the strong Markov property, given $W[0,\sigma]$, $(W(\sigma+t))_{t\ge 0}$ is distributed as an independent Brownian motion starting from $W(\sigma)$ until hitting $\S^1$. Therefore, for a.s.~such $W$ and any $\wt\sigma<T<\sigma$, $(W(\sigma+t))_{t\ge 0}$ will hit the interior of the complement of the infinite connected component of $\C\setminus W[0,T]$. In particular, $W$ is a.s.~in the interior $\textrm{int}(\mathcal G^c)$ of the set $\mathcal G^c$. Hence we have $\BE(\A_\tau)^\#[\mathcal G^c\setminus \textrm{int}(\mathcal G^c)]=0$, which implies $\BE(\A_\tau)^\#[\partial\mathcal{G}]=0$ since $\mathcal{G}^c$ is closed (and thus $\partial\mathcal{G}=\mathcal G^c\setminus \textrm{int}(\mathcal G^c)$). This concludes the proof.
\end{proof}

\section{The $\SLE_{8/3,\D}^{\lp}$ measure via conformal welding}\label{sec:loop-app}

We provide a proof of Proposition~\ref{thm:weld-BA}. This is by removing the filled metric ball in the setting of Proposition \ref{prop-weld-loop}, following the approach in~\cite[Proposition 6.10, Theorem 7.4]{ARS2022moduli}.

\begin{proof}[Proof of Proposition~\ref{thm:weld-BA}]
Let $(\widehat{\C},\phi,0,\infty)$ be an embedding of a sample from $\QS_2$. Let $d_\phi$ be the $\sqrt{8/3}$-LQG metric associated to $\phi$~\cite{DDDF-tightness,gm-uniqueness}. As demonstrated in \cite{lqg-tbm1}, there exists constants $c_1,c_2>0$, such that the law of the metric measure space $(\widehat{\C},\phi,0,\infty,c_1d_{\phi},c_2\mu_{\phi})$ is given by a constant multiple of $\BS_2$, where $\BS_2$ is the free Brownian sphere with two marked points~\cite{LeGall13,Mie13}, see also \cite[Definition 6.1]{ARS2022moduli}.
In the following, we write $d:=c_1d_\phi$ for short.

Define $g(\ell):=\ell e^{-\frac{9}{2}\ell}$. Let $F$ be the event that $d(0,\infty)>1$. Restricted on $F$, let $B^{\bullet}(\infty,1)$ be the filled metric ball with radius 1 centered at $\infty$, and $\mathcal{L}$ be the boundary length of $\partial B^{\bullet}(\infty,1)$. By \cite[Proposition 6.2, Theorem 6.5 and Lemma 6.9]{ARS2022moduli}, under the law of $\frac{1_F}{g(\mathcal{L})}\QS_2$, $(\widehat{\C}\backslash B^{\bullet}(\infty,1),\phi,0)/\mathord\sim_\gamma$ is a constant multiple of $\QD_{1,0}$. Sample $\eta$ from $\SLE_{8/3}^{\sep}$ independently, and let $E$ be the event that $d(\eta,\infty)>1$ (i.e.~$\eta\subset\widehat{\C}\backslash B^{\bullet}(\infty,1)$). According to the conformal restriction property of $\SLE_{8/3}$ loop measure, under $\frac{1_E}{g(\mathcal{L})}\QS_2\otimes\SLE_{8/3}^{\sep}$, the law of $(\widehat{\C}\backslash B^{\bullet}(\infty,1),\phi,\eta,0)/\mathord\sim_\gamma$ is a constant multiple of $\QD_{1,0}\otimes\SLE_{8/3}^{\lp}$, where $\QD_{1,0}\otimes\SLE_{8/3}^{\lp}:=\int_0^\infty\QD_{1,0}(\ell)\otimes\SLE_{8/3}^{\lp}d\ell$.

Let $D_\eta$ be the bounded connected component of $\widehat{\C}\backslash\eta$ and $\ell$ be the quantum length of $\eta$.
Note that on the other hand, by Proposition \ref{prop-weld-loop} and \cite[Proposition 6.10]{ARS2022moduli}, under $\frac{1_E\QS_2\otimes\SLE_{8/3}^{\sep}}{\ell|\QD_{1,0}(\ell)|g(\mathcal{L})}$, the law of $(\widehat{\C}\backslash(D_\eta\cup B^{\bullet}(\infty,1)),\phi)/\mathord\sim_\gamma$ is a constant multiple of $\BA$. Using Proposition~\ref{prop-weld-loop} again, we have that under $\frac{1_E}{g(\mathcal{L})}\QS_2\otimes\SLE_{8/3}^{\sep}$, the law of $(\widehat{\C}\backslash B^{\bullet}(\infty,1),\phi,\eta,0)/\mathord\sim_\gamma$ is a constant multiple of $\int_0^\infty\ell\Weld(\QD_{1,0}(\ell),\BA(\ell,\cdot))d\ell$, where $\BA(\ell,\cdot):=\int_0^\infty\BA(\ell,L)dL$. The result then follows.
\end{proof}

\bibliographystyle{alpha}

\end{document}